\documentclass{amsart}

\usepackage{amsmath}
\usepackage{amssymb}
\usepackage{tabularx}
\usepackage{enumerate}
\usepackage{tikz} 
\usepackage{hyperref}
\usepackage{a4wide}
\usepackage{enumitem}

\newtheorem{lemma}{Lemma}[section]
\newtheorem{remark}[lemma]{Remark}
\newtheorem{theorem}[lemma]{Theorem}

\newtheorem{proposition}[lemma]{Proposition}

\title{Counting Phylogenetic Networks of Level $1$ and $2$}

 \author[M. Bouvel]{Mathilde Bouvel}
       \address[MB]{Institut f\"ur Mathematik, Universit\"at Z\"urich, Winterthurerstr. 190, CH-8057 Zürich, Switzerland}
       \email{mathilde.bouvel@math.uzh.ch}
       
 \author[P. Gambette]{Philippe Gambette}
       \address[PG]{Universit{\'e} Paris-Est, LIGM (UMR 8049), UPEM, CNRS, ESIEE, ENPC, F-77454, Marne-la-Vall{\'e}e, France}
       \email{philippe.gambette@u-pem.fr}       
       
 \author[M. Mansouri]{Marefatollah Mansouri}
       \address[MM]{Technische Universit{\"a}t Wien, Department of Discrete Mathematics and Geometry, Wiedner Hauptstra{\ss}e 8-10/104, A-1040 Wien, Austria.}
       \email{marefatollah.mansouri@tuwien.ac.at}    

\begin{document}
\maketitle              

\begin{abstract}
Phylogenetic networks generalize phylogenetic trees, and 
have been introduced in order to describe evolution in the case of transfer of genetic material
between coexisting species. 
There are many classes of phylogenetic networks, 
which can all be modeled as families of graphs with labeled leaves.
In this paper, we focus on 
rooted and unrooted level-$k$ networks and provide enumeration
formulas (exact and asymptotic) for rooted and unrooted level-1 and level-2 phylogenetic networks 
with a given number of leaves.
We also prove that the distribution of some parameters of these networks 
(such as their number of cycles) are asymptotically normally distributed. 
These results are obtained by first providing a recursive description 
(also called combinatorial specification) of our networks, 
and by next applying classical methods of enumerative, symbolic and analytic combinatorics.
\end{abstract}


\medskip


\section{Introduction}

Phylogenetic networks generalize phylogenetic trees introducing 
reticulation vertices, which have two parents, and represent 
ancestral species resulting from the transfer of genetic material
between coexisting species, through biological processes such
as lateral gene transfer, hybridization or recombination.
More precisely, binary phylogenetic networks are usually defined
as rooted directed acyclic graphs with exactly one root,
tree vertices having one parent and 2 children,
reticulation vertices having 2 parents and one child and 
labeled leaves. The leaves are bijectively labeled by 
a set of taxa, which correspond to currently living species.

As for trees, phylogenetic networks can be rooted or unrooted.
Ideally, phylogenetic networks should be rooted, 
the root representing  the common ancestor of all taxa 
labeling the leaves.
But several methods which reconstruct phylogenetic networks,
such as combinatorial~\cite{HMSW2018,IerselMoulton2018}, distance-based~\cite{BDM2012,WTM2014} or 
parsimony-based methods~\cite{PosadaCrandall2001,LabarreVerwer2014},
do not produce inherently rooted networks, but provide
unrooted networks where tree vertices
and reticulation vertices cannot be distinguished.

An important parameter that allows to measure the complexity of a phylogenetic networks is its level. 
Phylogenetic trees are actually phylogenetic networks of level $0$, 
and the level of a network $N$ measures ``how far from a tree'' $N$ is. 

The problem of enumerating (rooted or unrooted) trees is a very classical one in enumerative combinatorics. 
Solving this problem actually led to general methods for enumerating other tree-like structures, 
where generating functions play a key role. 
We will review some of these methods in Section~\ref{sec:GF}. 
These methods have successfully been used by Semple and Steel~\cite{SempleSteel2006} 
to enumerate two families of phylogenetic networks, 
namely unicyclic networks and unrooted level-$1$ networks (also called \emph{galled trees}). 
Their results include an equation defining implicitly the generating function 
for unrooted level-$1$ networks (refined according to two parameters), 
which yields a closed formula for the number of  unrooted level-$1$ networks 
with $n$ (labeled) leaves, $k$ cycles, and a total of $m$ edges (also called arcs) across all the cycles. 
An upper bound on the number of unlabeled galled trees is also provided in~\cite{CHT2018}.
Other counting results have been more recently obtained 
on other families of phylogenetic networks, for example on so-called normal and tree-child networks~\cite{MSW2015,FGM2018} and on galled networks~\cite{GRZ2018}.

In this paper, we extend the results of Semple and Steel in several ways. 
First, about unrooted level-$1$ networks, we provide an asymptotic estimate 
of the number of such networks with $n$ (labeled) leaves. 
We also prove that the two parameters considered by Semple and Steel 
are asymptotically normally distributed. 
Second, we consider rooted level-$1$ networks, 
whose enumeration does not seem to have been considered so far in the literature. 
For these networks, we provide a closed formula counting them by number of leaves, 
together with an asymptotic estimate, 
and a closed formula for their enumeration refined by two parameters 
(the number of cycles and number of edges across all the cycles). 
Moreover, we show that these two parameters are asymptotically normally distributed. 
Finally, we consider both unrooted and rooted level-$2$ networks.
Similarly, we provide in each case exact and asymptotic formulas for their enumeration, 
and prove asymptotic normality for some parameters of interest, 
namely: the number of bridgeless components of strictly positive level, 
and the number of edges across them. 
These parameters are a generalization for level-$k$ ($k>2$) of those considered by Semple and Steel for level-1, 
in the sense that they quantify how different from a tree these phylogenetic networks are. 

The results of this paper rely heavily on analytic combinatorics~\cite{FlajoletSedgewick2008}. This framework can also be used to derive uniform random generators (for example with the recursive method~\cite{FlajoletEtAl1994} or with a Boltzmann sampler~\cite{DuchonEtAl2004}) directly from the specifications of the classes of phylogenetic networks given below. This could be useful for applications in bioinformatics, especially to generate simulated data in order to evaluate the speed or the quality of the output of algorithms dealing with phylogenetic networks.

Table~\ref{tab:overview} provides an overview of our results, and of where they can be found in the paper. 
\begin{table}[ht]
    \centering
    \begin{tabular}{|l|l|l|l|l|}
    \hline
        Type of network & Unrooted, & Rooted, & Unrooted, & Rooted, \\
         & level-$1$ & level-$1$ & level-$2$ & level-$2$ \\
    \hline
        Letter $\mathcal{X}$ denoting the class & 
        $\mathcal{G}$ (\textbf{g}alled) & $\mathcal{R}$ (\textbf{r}ooted) & $\mathcal{U}$ (\textbf{u}nrooted) & $\mathcal{L}$ (\textbf{l}ast) \\
    \hline
        Eq. for the EGF $X(z)$ & Thm.~\ref{thm:GF_Unrooted_1} \hfill $(*)$& Thm.~\ref{thm:GF_Rooted_1}& Thm.~\ref{thm:GF_Unrooted_2}&Thm.~\ref{thm:GF_rooted_level2} \\
    \hline
        Exact formula for $x_n$ & Thm.~\ref{thm:GF_Unrooted_1} \hfill $(*)$& Prop.~\ref{prop:enum_Rooted_1}& Prop~\ref{eq:exact_unrooted2} & Prop.~\ref{exactrooted2}\\
    \hline
        Asymptotic estimate of $x_n$ & Prop.~\ref{prop:asym_Unrooted_1}& Prop.~\ref{prop:asym_Rooted_1}& Prop.~\ref{prop:asym_Unrooted_2}& Prop.~\ref{asy_rooted2} \\
    \hline
        Eq. for the multivariate EGF & Eq.~\eqref{eq:unrooted1multivariate} \hfill $(*)$& Eq.~\eqref{eq:rooted1multivariate}& Eq.~\eqref{eq:multi_unrooted2} & Eq.~\eqref{eq:multi_rooted2}\\
    \hline
        Asymptotic normality & Prop.~\ref{prop:normaldistribunrooted1} & Prop.~\ref{prop:normaldistribrooted2}& Prop.~\ref{prop:normaldistribunrooted2} & Prop.~\ref{asymultirooted2}\\
    \hline
    \end{tabular}
    
    \smallskip
    
    \caption{Overview of our main results. EGF means exponential generating function. 
    The results marked with $(*)$ also appear in the work of Semple and Steel~\cite{SempleSteel2006}.
    In addition, refined enumeration formulas for unrooted and rooted level-1 networks are provided in 
    \cite[Thm. 4]{SempleSteel2006} and Prop.~\ref{prop:refinedEnumerationRooted1} respectively. 
    (Although the proof method applies to obtain such formulas for level-2 as well, 
    the computations would however be rather intricate, 
    and the interest \emph{a priori} of the formulas so obtained questionable, hence our choice not to do it.)
    \label{tab:overview}}
\end{table}

For the reader's convenience, we also include in this introduction 
the beginnings of the enumeration sequences of the four types of networks we consider, as well as their asymptotic behavior -- see Table~\ref{TableNumbers}. 
We have added these sequences to the OEIS~\cite{OEIS}, and we also include in Table~\ref{TableNumbers} their OEIS reference. 
We also provide in the supplementary material the code in the DOT language, 
as well as a visualization with GraphViz, 
of the 15 unrooted level-1 networks on 4 leaves,
the 3 rooted level-1 networks on 2 leaves, 
the 6 unrooted level-2 networks on 3 leaves
and the 18 rooted level-2 networks on 2 leaves.

\begin{table}[ht]
\centering
\begin{tabular}{c||c|c|c|c}
$n$ & $g_{n-1}$ & $r_n$ & $u_{n-1}$ & $\ell_n$ \\
\hline
1 & 0 & 1 & 0 & 1\\
2 & 1 & 3 & 1 & 18\\
3 & 2 & 36 & 6 & 1 143\\
4 & 15 & 723 & 135 & 120 078 \\
5 & 192 & 20 280 & 5 052 & 17 643 570  \\
6 & 3 450 & 730 755 &264 270 & 3 332 111 850\\
\hline
as $n \to \infty$ & $c_1 \approx 0.20748$ & $c_1 \approx 0.1339$ & $c_1 \approx 0.07695$ & $c_1 \approx 0.02931$\\
$x_n \sim c_1 c_2^n n^{n-1}$ with & $c_2 \approx 1.89004$ &  $c_2 \approx 2.943$ & $c_2 \approx 5.4925$ & $c_2 \approx 15.4333$\\
\hline
OEIS reference & A328121 & A328122 & A333005 & A333006 \\
\end{tabular}
\smallskip
\caption{The numbers of rooted and unrooted level-1 or level-2 networks on $n$ leaves. \label{TableNumbers}}
\end{table}

\medskip

The remainder of the article is organized as follows. 
First, Section~\ref{sec:phylo_intro} recalls definitions and properties of phylogenetic networks 
that are important for our purpose. 
Next, Section~\ref{sec:GF} reviews some methods of enumerative and analytic combinatorics 
which we will apply to solve the enumeration of our networks in the following sections. 
Precisely, Sections~\ref{sec:unrooted1} and~\ref{sec:rooted1} deal with level-$1$ networks, unrooted and rooted, while Sections~\ref{sec:unrooted2} and~\ref{sec:rooted2} focus on level-$2$ networks. 

\section{Some properties of phylogenetic networks}\label{sec:phylo_intro}

\subsection{Rooted binary phylogenetic networks}

In graph theory, a \textit{cut arc} or \textit{bridge} of a directed graph 
 $G$ is an arc whose deletion disconnects $G$. A \textit{bridgeless component} of a graph 
 is a maximal induced subgraph of $G$ without cut arcs.

We define a \textit{binary rooted phylogenetic network} $N$
on a set $X$ of leaf labels, for $|X|\geq 2$ 
as a directed acyclic 
 graph having: 
\begin{enumerate}
    \item exactly one \textit{root}, that is an in-degree-0 out-degree-2 vertex; 
    \item \textit{leaves}, that is in-degree-1 out-degree-0 vertices which are bijectively labeled by elements of $X$;
    \item \textit{tree vertices}, that is in-degree-1 out-degree-2 vertices;
    \item \textit{reticulation vertices}, that is in-degree-2 out-degree-1 vertices; \\
    and such that 
    \item for each bridgeless component $B$ of $N$, 
there exist at least two cut arcs of $N$ whose tail\footnote{The \emph{tail} of an arc 
is by definition its starting point. Its arrival point is called \emph{head}.} belongs to $B$ and whose head does not belong to $B$.
\end{enumerate}
A binary rooted phylogenetic network $N$ on a singleton ${x}$ is a single vertex labeled by $x$.

As illustrated in Fig.~\ref{fig:rooted-unrooted}, a binary rooted phylogenetic network $N$ is said to be
\textit{level-$k$} (or called a \textit{level-$k$ network} for short)
if the number of reticulation vertices contained in any
bridgeless component of $N$ is less than or equal to $k$. 
In a level-1 network $N$, each bridgeless component $B$ having at least two vertices 
consists of the union of two directed paths, which start and end at the same vertices, 
called \emph{source} and \emph{sink} respectively. 
The source is actually either the root of $N$, or the head of a cut arc of $N$, 
and the sink is the unique reticulation vertex of $B$. 
Such bridgeless components are called \emph{cycles}. 

Note that variations on the definition of rooted binary phylogenetic networks 
are around in the literature, and a few comments on our choice of definition are in order. 
Like in most publications about phylogenetic networks,  our definition of binary rooted phylogenetic networks
does not allow multiple arcs.
As our goal is to study a model of binary phylogenetic
networks that could be counted if their number of leaves and level 
are fixed, condition (5) is necessary to ensure that
there are finitely many phylogenetic networks 
with a given number of leaves and level. Note that this restriction has already appeared in the literature under the name ``networks with no redundant biconnected components''~\cite{IerselMoulton2014} or ``with no redundant blobs''~\cite{GIKPS2016}.
Indeed, without it, such networks have an unbounded number of vertices: 
this can be seen by replacing any cut arc of the network by a sequence of networks isomorphic to the one in Fig.~\ref{rootedlevel2}($2a$), which has only one incoming cut-arc and one outgoing cut-arc.

Similarly in some algorithmic-oriented papers about phylogenetic
networks, bridgeless components with three vertices and
two outgoing arcs are 
forbidden because the information needed to distinguish
those components from simple tree vertices also connected
with two outgoing arcs is not available in the input data. In the perspective of counting those
objects we do not impose this restriction. But it could easily
be added to our combinatorial descriptions and formulas below, 
to be taken into account if needed.

\subsection{Unrooted binary phylogenetic networks}

Now, we extend the latter definition to unrooted phylogenetic networks.
A \textit{cut-edge} or \textit{bridge} of an undirected graph 
 $G$ is an edge whose removal disconnects the graph. A \textit{bridgeless component} of a graph 
 $G$ is a maximal induced subgraph of $G$ without cut-edges.

An \textit{unrooted binary phylogenetic network} $N$ on a set $X$ of at least 2 leaf labels is a loopless (undirected) graph whose vertices have either degree $3$ (\textit{internal vertices}) or degree $1$ (\textit{leaves}), such that its set $L(N)$ of leaves is bijectively labeled by $X$ and such that for each bridgeless component $B$ of $N$ having strictly more than one vertex, the set of cut-edges incident with some vertex of $B$ has size at least 3.
An unrooted binary phylogenetic network $N$ on a singleton ${x}$ is a single vertex labeled by $x$.
An unrooted binary 
phylogenetic tree is an unrooted binary phylogenetic network 
with no bridgeless component containing strictly 
more that one vertex.
An unrooted binary phylogenetic network 
is said to be \textit{level-$k$}
(or called an \textit{unrooted level-$k$ network} for short)
if an unrooted binary phylogenetic tree 
can be obtained by first removing at most $k$ edges per bridgeless component, and then, for each degree-2 vertex, contracting the edge between this vertex and one of its neighbours.
We denote by \emph{cycles}
the bridgeless components of unrooted level-1 networks 
having strictly more than one vertex. 
(Indeed, they are just cycles -- of size at least $3$ -- in the graph-theoretical sense.)

Note that given a rooted level-$k$ network $N$ on $n$ leaves,
we can obtain an unrooted binary phylogenetic network $N'$ on $n+1$ leaves with the following \emph{unrooting procedure}: 
add a vertex adjacent to the root of $N$, labeled with an extra leaf label (usually denoted $\#$),
and ignore all arc directions.
Theorem~1 of~\cite{GBP2012} implies in addition 
that the network $N'$ so obtained is an unrooted level-$k$ network. 
This unrooting procedure which consists of building an unrooted level-$k$ network from a rooted level-$k$ network, illustrated in Fig.~\ref{fig:rooted-unrooted},
can be reversed (see Lemma 4.13 of~\cite{JJEIS2018}), although not in a unique fashion. 
Indeed, given an unrooted level-$k$ network $N'$ on $n+1$ leaves, 
it is possible to choose any leaf and delete it, making its neighbour become
the root $\rho$ of a rooted level-$k$ network $N$ obtained by: 
\begin{enumerate}
    \item placing the bridgeless component $B$ containing $\rho$ at the top;
    \item orienting downwards all the cut-edges incident with vertices of $B$; 
    \item choosing the tail $t$ of one of these cut arcs as the sink of $B$; 
    \item computing an $\rho$-$t$ numbering~\cite{LEC1967} on the vertices of $B$ if there are more than one, that is labeling vertices of $B$ with integers from 1 to the number $n_B$ of vertices of $B$, such that the labels of $\rho$ and $t$ are respectively 1 and $n_B$ and such that any vertex of $B$ except $\rho$ and $t$ is adjacent both to a vertex with a lower label and a vertex with a higher label;
    \item orienting each edge of $B$ by choosing its vertex with the lower label as the tail;
\\
and 
    \item moving downwards into the network, recursively applying this procedure on all
other bridgeless components.
\end{enumerate}
This correspondence is not one-to-one because of the choices of the leaf which is deleted, 
and most importantly because of the choices of sinks in step 3 above. 
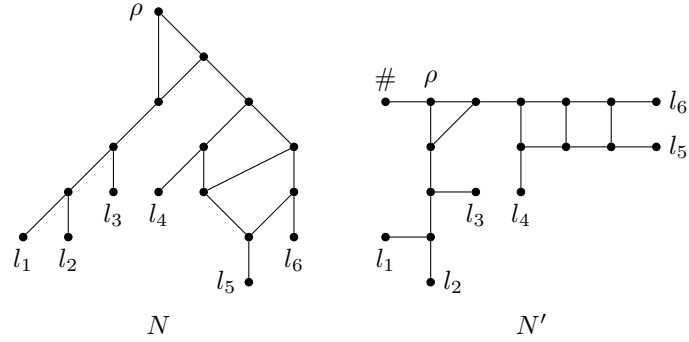
\begin{figure}[h]
	\begin{center}
\begin{tabular}{ccc}
\begin{tikzpicture}
\begin{scope}[scale=0.3]
\draw (-1,12) node {$\rho$};
\draw (0,12) node [inner sep=1pt,circle,fill,draw] (root) {};
\draw (2,10) node [inner sep=1pt,circle,fill,draw] (a) {};
\draw (0,8) node [inner sep=1pt,circle,fill,draw] (b) {};
\draw (-2,6) node [inner sep=1pt,circle,fill,draw] (c) {};
\draw (-4,4) node [inner sep=1pt,circle,fill,draw] (d) {};
\draw (4,8) node [inner sep=1pt,circle,fill,draw] (e) {};
\draw (2,6) node [inner sep=1pt,circle,fill,draw] (f) {};
\draw (2,4) node [inner sep=1pt,circle,fill,draw] (g) {};
\draw (4,2) node [inner sep=1pt,circle,fill,draw] (h) {};
\draw (6,6) node [inner sep=1pt,circle,fill,draw] (i) {};
\draw (6,4) node [inner sep=1pt,circle,fill,draw] (j) {};
\draw (-6,2) node [inner sep=1pt,circle,fill,draw] (l1) {};
\draw (-4,2) node [inner sep=1pt,circle,fill,draw] (l2) {};
\draw (-2,4) node [inner sep=1pt,circle,fill,draw] (l3) {};
\draw (0,4) node [inner sep=1pt,circle,fill,draw] (l4) {};
\draw (4,0) node [inner sep=1pt,circle,fill,draw] (l5) {};
\draw (6,2) node [inner sep=1pt,circle,fill,draw] (l6) {};
\draw (-6,1) node {$l_1$};
\draw (-4,1) node {$l_2$};
\draw (-2,3) node {$l_3$};
\draw (0,3) node {$l_4$};
\draw (3,0) node {$l_5$};
\draw (6,1) node {$l_6$};

\draw (root) -- (a);
\draw (a) -- (b);
\draw (root) -- (b);
\draw (b) -- (c);
\draw (c) -- (l3);
\draw (c) -- (d);
\draw (d) -- (l1);
\draw (d) -- (l2);
\draw (a) -- (e);
\draw (e) -- (f);
\draw (f) -- (l4);
\draw (f) -- (g);
\draw (g) -- (h);
\draw (h) -- (l5);
\draw (e) -- (i);
\draw (i) -- (g);
\draw (i) -- (j);
\draw (j) -- (l6);
\draw (j) -- (h);
\end{scope}
\end{tikzpicture}
& ~~~~ &
\begin{tikzpicture}
\begin{scope}[scale=0.3]
\draw (-6,9) node {$\#$};
\draw (-6,8) node [inner sep=1pt,circle,fill,draw] (lroot) {};
\draw (-4,9) node {$\rho$};
\draw (-4,8) node [inner sep=1pt,circle,fill,draw] (root) {};
\draw (-2,8) node [inner sep=1pt,circle,fill,draw] (a) {};
\draw (-4,6) node [inner sep=1pt,circle,fill,draw] (b) {};
\draw (-4,4) node [inner sep=1pt,circle,fill,draw] (c) {};
\draw (-4,2) node [inner sep=1pt,circle,fill,draw] (d) {};
\draw (0,8) node [inner sep=1pt,circle,fill,draw] (e) {};
\draw (0,6) node [inner sep=1pt,circle,fill,draw] (f) {};
\draw (2,6) node [inner sep=1pt,circle,fill,draw] (g) {};
\draw (4,6) node [inner sep=1pt,circle,fill,draw] (h) {};
\draw (2,8) node [inner sep=1pt,circle,fill,draw] (i) {};
\draw (4,8) node [inner sep=1pt,circle,fill,draw] (j) {};
\draw (-6,2) node [inner sep=1pt,circle,fill,draw] (l1) {};
\draw (-4,0) node [inner sep=1pt,circle,fill,draw] (l2) {};
\draw (-2,4) node [inner sep=1pt,circle,fill,draw] (l3) {};
\draw (0,4) node [inner sep=1pt,circle,fill,draw] (l4) {};
\draw (6,6) node [inner sep=1pt,circle,fill,draw] (l5) {};
\draw (6,8) node [inner sep=1pt,circle,fill,draw] (l6) {};
\draw (-6,1) node {$l_1$};
\draw (-3,0) node {$l_2$};
\draw (-2,3) node {$l_3$};
\draw (0,3) node {$l_4$};
\draw (7,6) node {$l_5$};
\draw (7,8) node {$l_6$};

\draw (root) -- (lroot);
\draw (root) -- (a);
\draw (a) -- (b);
\draw (root) -- (b);
\draw (b) -- (c);
\draw (c) -- (l3);
\draw (c) -- (d);
\draw (d) -- (l1);
\draw (d) -- (l2);
\draw (a) -- (e);
\draw (e) -- (f);
\draw (f) -- (l4);
\draw (f) -- (g);
\draw (g) -- (h);
\draw (h) -- (l5);
\draw (e) -- (i);
\draw (i) -- (g);
\draw (i) -- (j);
\draw (j) -- (l6);
\draw (j) -- (h);
\end{scope}
\end{tikzpicture}
\\
$N$ & ~ & $N'$
\end{tabular}
\end{center}
\caption[]{
A rooted level-2 network $N$ (where all arcs are directed downwards) and the unrooted level-2 network $N'$ obtained by applying the unrooting procedure on $N$.\label{fig:rooted-unrooted}}

\end{figure}

\subsection{Decomposition of rooted and unrooted level-$k$ networks}

For any bridgeless component $B$ with $k_B\leq k$ reticulation vertices
of a rooted level-$k$ network $N$,
the directed multi-graph obtained by removing all outgoing arcs, and then 
contracting each arc from an in-degree-1 out-degree-1 vertex to its child,
is called a \emph{level-$k_B$ generator}~\cite{IKKSHB2009,GBP2009}. For each $k>0$, there exists a finite list of \emph{level-$k$ generators} which can be built from level-($k-1$) generators~\cite{GBP2009}.
Therefore, depending on the level $k_\rho$ of the bridgeless component $B_\rho$ of $N$ containing its root $\rho$, $N$ can be decomposed in the following way. It is either:
\begin{itemize}
    \item a single leaf if $k_\rho=0$ and $\rho$ has out-degree 0;
    \item a root $\rho$ being the parent of the root $\rho_1$ of a rooted level-$k$ network $N_1$ and of the root $\rho_2$ of a rooted level-$k$ network $N_2$ with disjoint sets of leaf labels, if $k_\rho=0$ and $\rho$ has out-degree $2$;
    \item a level-$k_\rho$ generator $G_\rho$ containing the root, with $0 < k_\rho \leq k$, whose arcs are subdivided to create new in-degree-1 out-degree-1 vertices, to which we add a set of cut arcs, whose
tails are the out-degree-0 vertices of $G_\rho$
and the newly created in-degree-1 out-degree-1 vertices, and whose heads are roots of rooted level-$k$ networks with disjoint sets of leaf labels. 
\end{itemize}

Similarly, for any bridgeless component $B$ of an unrooted level-$k$ network $N$, the multi-graph obtained by first removing all cut-edges incident with any vertex of $B$, 
then, for each degree-2 vertex, contracting the edge between this vertex and one of its neighbours, is
called an \emph{unrooted level-$k_B$ generator}~\cite{GBP2012,HMW2016}. An unrooted level-$k_B$ generator can also be defined as
a single vertex for $k_B=0$, as two vertices linked by a multiple edge for $k_B=1$, and as a 3-regular bridgeless multi-graph with $2k_B-2$ vertices for $k_B>1$ (Lemma 6 of~\cite{HMW2016}).
Therefore, by considering a leaf $l_{\#}$ of any unrooted level-$k$ network $N$ and the bridgeless component $B$ containing the vertex adjacent to this leaf, depending on the level $k_B$ of $B$, $N$ can be decomposed in the following way. 
\begin{itemize}
    \item If $k_B=0$ and $B$ consists of a single vertex of degree $1$ in $N$, then $N$ is just the leaf $l_{\#}$ adjacent to another leaf. 
    \item If $k_B=0$ and $B$ is not a single vertex of degree $1$ in $N$, then the leaf $l_{\#}$ is adjacent to a vertex $v$ of degree $3$ in $N$, such that the other two edges incident to $v$ are cut-edges. 
    $N$ is described by the edge between $l_{\#}$ and $v$, plus the two other edges incident with $v$, which are in turn identified with edges of two unrooted level-$k$ networks $N_1$ and $N_2$ with disjoint sets of leaf labels (not containing $\#$) in such a way that $v$ is identified with a leaf $l_{\#1}$ (resp. $l_{\#2}$) of $N_1$ (resp. $N_2$), removing the leaf labels of $l_{\#1}$ and $l_{\#2}$ during this identification.   
    \item Otherwise $0<k_B\leq k$. In this case, $N$ is described by taking a level-$k_{B}$ generator whose edges are subdivided to insert vertices, and then performing identification of these inserted vertices (in a same flavor as in the previous case). Specifically, one of these inserted vertices is identified with the neighbour of $l_{\#}$ in $N$, and all others are identified with leaves of unrooted level-$k$ networks with disjoint sets of leaf labels (not containing $\#$). Again, each leaf that is identified with another vertex looses its label during this identification.
\end{itemize}
These decompositions of rooted and unrooted level-$k$ networks will be the key to our counting results below.

\section{Generating functions: some basics tools and techniques}
\label{sec:GF}

This section summarizes some of the basics on combinatorial classes and their generating functions that we will use in our work. 
Our presentation follows closely \cite{FlajoletSedgewick2008} (although with much less details), 
and the reader interested to know more on the topic is referred to \cite[mainly Chapters I.5, II.1, II.5, VI.3, VII.3, VII.4]{FlajoletSedgewick2008}. 
The reader familiar with the classical tools of analytic combinatorics may safely skip this section.

\subsection{(Univariate) generating functions and counting}

Generally speaking, a \emph{combinatorial class} $\mathcal{C}$ is a set of discrete objects, equipped with a notion of size, such that for every integer $n$ there is a finite number of objects of size $n$ in $\mathcal{C}$. We denote by $\mathcal{C}_n$ the set of objects of size $n$ in $\mathcal{C}$, and by $c_n$ the cardinality of $\mathcal{C}_n$. Specifically in this paper, 
each combinatorial class we consider is a family of level-$k$ phylogenetic networks, 
and the size of such a network is its number of leaves.

Objects of size $n$ in $\mathcal{C}$  can be seen as an arrangement (following some rules to be made precise) of $n$ atoms, which are objects of size $1$. In our context, these atoms are the leaves of the networks, representing the current species, or \emph{taxa}. 
When the atoms constituting an object are distinguishable among themselves, 
the considered combinatorial objects are said to be \emph{labeled}\footnote{Although it is also very classical, the case of \emph{unlabeled} objects (with their corresponding \emph{ordinary} generating functions) will not be useful in our work, and is therefore omitted from our presentation.}. 
Because leaves of level-$k$ networks correspond to taxa, our networks are indeed labeled combinatorial objects.
Without loss of generality (\emph{i.e.}, up to relabeling), atoms in a labeled object of size $n$ are simply labeled by integers from $1$ to $n$, and we will take this convention in our work.

To a (labeled) combinatorial class $\mathcal{C}$, we can associate its \emph{exponential} \emph{generating function} $C(z) = \sum_{n\geq 0} c_n \frac{z^n}{n!}$, which is a formal power series in $z$ encapsulating the entire enumeration of $\mathcal{C}$. 

A \emph{specification} for a combinatorial class is an unambiguous description of the objects in the class using simpler classes and possibly the class itself. 
For instance, consider labeled rooted ordered binary trees, and define their size to be the number of their leaves. Such a tree is unambiguously described as being either a leaf or composed of a root to which a left and a right subtree are attached, which are themselves labeled rooted ordered binary trees, with a \emph{consistent relabeling} of their atoms. 
By this, we mean the following: considering two trees whose atoms are labeled by $\{1, \dots, k\}$ and  $\{1, \dots, k'\}$, we can build a tree using the first (resp. second) as left (resp. right) subtree; 
the atoms of this tree are labeled by $\{1, \dots, k+k'\}$, and need to be such that the relative order between the labels in the left (resp. right) subtree is preserved (and they may be in any such way). 
This specification for labeled rooted ordered binary trees can be formally written as follows: $\mathcal{B} =  \bullet$ \ $ \uplus $ \ {\begin{tikzpicture}[level distance=15mm,baseline=-10pt]
\begin{scope}[scale=0.4]
\tikzstyle{level 1}=[sibling distance=15mm]
\node[inner sep=0pt]{$\circ$}
  child {node[inner sep=0pt]{$\mathcal{B}$}}
  child {node[inner sep=0pt]{$\mathcal{B}$}}
;
\end{scope}
\end{tikzpicture}}, 
where $\bullet$ represents a leaf (contributing $1$ to the size of the object) and $\circ$ represents an internal node (which contributes $0$ to the size). 

Specifications describing (labeled) combinatorial classes can be translated into equations satisfied by the corresponding (exponential) generating functions. The precise statement that we refer to is \cite[Theorem II.1]{FlajoletSedgewick2008}.
The following proposition summarizes the simplest cases of this translation, 
which we will often use later in this paper.
\begin{proposition}[Dictionary]
Let $\mathcal{A}$ and $\mathcal{B}$ be two labeled combinatorial classes. 
Denote by $A(z)$ and $B(z)$ their respective exponential generating functions. 
Then the generating function of the class which is the disjoint union of $\mathcal{A}$ and $\mathcal{B}$ 
(resp. the Cartesian product of $\mathcal{A}$ and $\mathcal{B}$) is $A(z) +B(z)$ (resp. $A(z)\cdot B(z)$). 
In addition, if $\mathcal{A}$ contains no object of size $0$, the class which consists of sequences of objects of $\mathcal{A}$ (\emph{i.e.}, $m$-tuples of objects of $\mathcal{A}$, for any $m \geq 0$) has generating function $\frac{1}{1-A(z)}$.
\label{dictionary}
\end{proposition}
On the previous example of binary trees, it follows from the above proposition that the corresponding generating function satisfies $B(z) = z+ B(z)^2$.

The next step is to have access to the enumeration sequence $(c_n)$ of a class $\mathcal{C}$ from an equation satisfied by the generating function $C(z)$ of $\mathcal{C}$. 
A possible way, especially in the case of tree-like objects, is to appeal to the \emph{Lagrange inversion formula} (\cite[Theorem A.2]{FlajoletSedgewick2008}). 
To state it, we introduce the notation $[z^n] C(z)$ to denote the $n$-th coefficient of the series $C(z)$; 
that is to say, writing $C(z) = \sum_{n\geq 0} \frac{c_n}{n!} z^n$, we have $[z^n] C(z) =\frac{c_n}{n!}$, or equivalently $c_n = n! \cdot [z^n] C(z)$. 

The Lagrange inversion formula is as follows. 

\begin{proposition}[Lagrange inversion formula]
Assume that a generating function $C$ satisfies an equation of the form $C(z) = z \phi(C(z))$ for $\phi(z) = \sum_{n\geq 0} \phi_n z^n$ a formal power series such that $\phi_0 \neq 0$. Then, we have:
$$ 
[z^n] C(z) = \frac{1}{n} [z^{n-1}] \phi(z)^n\textrm{.}
$$
\end{proposition}

Even though defined as formal power series, it is often useful to consider that generating functions are analytic functions of the complex variable $z$, in a small disk of convergence around the origin of the complex plane. This sometimes allows to find a closed form for the generating function in its disk of convergence, but not always. Even in this least favorable case, it enables to inherit fundamental results from complex analysis, which can be used for the purpose of enumerating combinatorial objects. In particular, we have in this tool box the \emph{Singular Inversion Theorem} (Theorem VI.6 of~\cite{FlajoletSedgewick2008}), 
which allows to derive asymptotic estimates of the coefficients of generating functions.

\begin{theorem}[Singular Inversion Theorem]
Let $C(z)$ be a generating function such that $C(0)=0$, satisfying the equation $C(z) = z \phi(C(z))$ for $\phi(z) = \sum_{n\geq 0} \phi_n z^n$ a power series such that $\phi_0 \neq 0$, all $\phi_n$ are non-negative real numbers, and $\phi(z) \neq \phi_0 + \phi_1 z$. 
Denote by $R$ the radius of convergence of $\phi$ at $0$. Assume that $\phi$ is analytic at $0$ (so that $R>0$), that the \emph{characteristic equation} $\phi(z) - z \phi'(z) = 0$ has a solution $\tau \in (0,R)$ (that is necessarily unique), and that $\phi$ is aperiodic\footnote{Aperiodicity is needed only for the third item below. The definition of aperiodicity is omitted from this paper, and can be found in~\cite[Definition IV.5]{FlajoletSedgewick2008}. A sufficient condition for a power series to be aperiodic (which applies to all examples considered in this paper), is to have $\phi_n >0$ for all $n$.\label{footnote:aperiodic}}. Then the followings hold:
\begin{itemize}
 \item $\rho = \frac{\tau}{\phi(\tau)}$ is the radius of convergence of $C$ at $0$;
 \item near $\rho$, $C(z) \sim \tau - \sqrt{\frac{2\phi(\tau)}{\phi''(\tau)}} \sqrt{1-\frac{z}{\rho}}$;
 \item when $n$ grows, $[z^n] C(z) \sim \sqrt{\frac{\phi(\tau)}{2\phi''(\tau)}} \frac{\rho^{-n}}{\sqrt{\pi n^3}}$.
\end{itemize}
\label{thm:singular_inversion}
\end{theorem}

\subsection{Multivariate generating functions and estimating parameters}

Until now, our generating functions had only a single variable, $z$, recording the size of the objects we were counting. 
We now consider \emph{multivariate} generating functions, where additional variables ($x$, $y$, \dots) record the value of other parameters of our objects. 
In our cases, we will consider at most two such parameters, which are numbers of certain ``substructures'' occurring in our objects. 
Namely, denoting $c_{n,k,m}$ the number of objects of size $n$ in the combinatorial class $\mathcal{C}$ 
such that the first parameter has value $k$ and the second has value $m$, the multivariate exponential generating function we consider is $C(z,x,y) = \sum_{n,k,m} \tfrac{c_{n,k,m}}{n!} z^nx^ky^m$. 

To continue our earlier example of binary trees, we could consider one additional parameter, 
which is the number of internal nodes. (Of course,we are aware that the number of internal nodes is always the number of leaves -- \emph{i.e.}, the size -- minus one; but we keep this example just to illustrate definitions and tools available.)
The coefficient of $z^nx^k$ in the generating function $B(z,x)$ is then the number of binary trees with $n$ leaves and $k$ internal nodes, divided by $n!$. 

The ``dictionary'' translating combinatorial specifications to equations satisfied by the generating function extends to multivariate series, and our earlier specification 
$\mathcal{B} =  \bullet$ \ $ \uplus $ \ {\begin{tikzpicture}[level distance=15mm,baseline=-10pt]
\begin{scope}[scale=0.4]
\tikzstyle{level 1}=[sibling distance=15mm]
\node[inner sep=0pt]{$\circ$}
  child {node[inner sep=0pt]{$\mathcal{B}$}}
  child {node[inner sep=0pt]{$\mathcal{B}$}}
;
\end{scope}
\end{tikzpicture}} 
gives $B(z,x) = z + x B(z,x)^2$. 

Here again, the Lagrange inversion formula may be used to derive a closed formula for the coefficients $c_{n,k,m}$. 
Indeed, assuming that our multivariate exponential generating function $C(z,x,y)$ satisfies an equation of the form $C(z,x,y) = z \phi(C(z),x,y)$ for $\phi(z,x,y) = \sum_{n\geq 0} \phi_n(x,y) z^n$ a formal power series such that $\phi_0 \neq 0$, then we have:
$$ 
\tfrac{c_{n,k,m}}{n!} = [z^nx^ky^m] C(z,x,y) = \frac{1}{n} [z^{n-1}x^ky^m] \phi(z,x,y)^n\textrm{.}
$$

Moreover, under some hypotheses, the following theorem (see \cite[Theorem 2.23]{Paperback}) 
allows to prove that the considered parameters are asymptotically normally distributed. 
The notation used in the statement of this theorem is as follows: 
if $F$ is a function of several variables, including $v$, $F_v$ denotes the partial derivative of $F$ with respect to $v$; 
as usual, $\mathbb{E}$ and $\mathbb{V}ar$ respectively denote expectation and variance; 
$\mathcal{N}(0,1)$ is the standard normal distribution; 
and $\xrightarrow{d}$ denotes convergence in distribution. 

\begin{theorem}\label{Drmota}
Assume that $ C(z, x) $ is a power series that is the (necessarily unique and analytic) solution of the functional equation $ C = F (C,z, x) $, where $ F (C,z, x) $ satisfies the following assumptions:
$F(C,z,x)$ is analytic in $C$, $z$ and $x$ around $0$, 
$F(C,0,x)=0$, 
$F(0,z,x) \neq 0$, 
and all coefficients $[z^n C^m] F(C,z,1)$ are real and non-negative. 

Assume in addition that the region of convergence of $F(C,z,x)$ is large enough 
for having non-negative solutions $z=z_0$ and $C=C_0$ of the system of equations 
\begin{align*} 
C&=F (C, z, 1)
\\
1&=F_{C} (C, z, 1)
\end{align*}
with $F_z(C_0,z_0,1) \neq 0$ and $F_{CC}(C_0,z_0,1) \neq 0$.

Then, if $X_n$ is a sequence of random variables such that 
\[
\mathbb{E} x^{X_n}= \dfrac{[z^n]C(z,x)}{[z^n]C(z,1)},
\]
then $X_n$ is asymptotically normally distributed. 

More precisely, setting\\
\qquad\qquad  $ \mu  =\dfrac{F_x}{z_{0} F_z} $
{\fontsize{8}{10}\selectfont
\begin{align*}
\sigma^{2}& = \mu +\mu^{2}+\dfrac{1}{z_{0}F_{z}^{3}F_{CC}}  \Big( F_{z}^{2}(F_{CC}F_{xx}-F_{Cx}^{2})-2F_{z}F_{x}(F_{CC}F_{zx}-F_{Cz}F_{Cx})+F_{x}^{2}(F_{CC}F_{zz}-F_{Cz}^{2}) \Big)
\end{align*}
}
where all partial derivatives are evaluated at the point $ (C_0 , z_0 , 1) $, we
have
\[
\mathbb{E} {X_n}= \mu n+ O(1) \;\;\;\;\;\; and \;\;\;\;\;\;\;\;\;\mathbb{V}ar{X_n}=\sigma^{2} n+O(1)
\]
and if $ \sigma^{2} > 0 $ then

\[
\dfrac{X_{n}-\mathbb{E} {X_n}}{\sqrt{\mathbb{V}ar{X_n}}}\xrightarrow{d} \mathcal{N}(0,1).
\]
\end{theorem}

\subsection{Implementation and note about computations}
Some of the computations used to obtain the results of this paper were programmed in Maple. A companion Maple document is available from the authors webpage\footnote{at \url{http://user.math.uzh.ch/bouvel/publications/BouvelGambetteMansouri_Version2_WithoutMultipleEdges.mw}}.

We also point out to the interested reader that a first version of this article 
was  considering a variant of the model of level-$k$ phylogenetic networks, 
where multiple (\emph{i.e.} parallel) edges are allowed. 
The counting results for this alternate model of course differ (starting from level $2$), 
and can be found in~\cite{ArxivV2}, again with an associated Maple document\footnote{available at  \url{http://user.math.uzh.ch/bouvel/publications/BouvelGambetteMansouri_Version1_WithMultipleEdges.mw}}.
Similarly, these files can easily be used to adapt the computations in case other restrictions are imposed on the structure of level-1 or level-2 phylogenetic networks, for example if \emph{tiny cycles}, defined in~\cite{HIMSW2017} as bridgeless components with exactly three vertices, are not allowed.

\section{Counting unrooted level-1 networks}\label{sec:unrooted1}

\subsection{Generating function and exact enumeration formula}

Unrooted level-1 networks (also called unrooted galled trees) have been enumerated in~\cite{SempleSteel2006}. The enumeration does not only consider the number of leaves of the galled trees, but is refined according to two parameters: the number of cycles (\emph{i.e.}, level-1 generators) and the total number of edges which are part of a cycle (that we will call \emph{inner edges}). We only reproduce in Theorem~\ref{thm:GF_Unrooted_1} a simplified version of the results of~\cite{SempleSteel2006}, taking into account the number of leaves only.

\begin{theorem}
For any $n\geq 0$, let $g_n$ denote the number of unrooted level-1 networks with $(n+1)$ leaves, and denote by $G(z) = \sum_{n\geq 0} g_n \frac{z^n}{n!}$ the corresponding generating function.
Then $G$ satisfies the following equation:
$$ 
G(z) = z + \frac{1}{2} G(z)^2 + \frac{1}{2} \frac{G(z)^2}{1-G(z)}
\textrm{,}
$$
or equivalently 
$$
G(z) = z\phi(G(z)) \textrm{ with } \phi(z) = \frac{1}{1-\frac{1}{2}z(1+\frac{1}{1-z})}
\textrm{.}
$$

Moreover, for any $n\geq 0$, let $g_n$ denote the number of unrooted level-1 networks with $(n+1)$ leaves. We have:
\begin{equation}
g_n = \frac{(2n-2)!}{2^{n-1}(n-1)!} + \sum_{1 \leq i \leq k \leq n-1} \frac{(n+i-1)! (n+k-i-2)!}{k! (k-1)! (i-k)! (n-i-1)!}2^{-i}
\textrm{.}\label{eq:enum_Unrooted_1}
\end{equation}
\label{thm:GF_Unrooted_1}
\end{theorem}

Notice that even if the formulas seem different, Eq.~\eqref{eq:enum_Unrooted_1} can be recovered from Theorem~4 of~\cite{SempleSteel2006} by summing over $k$ and $m$ and performing the change of variable $m = n-i+3k-1$. 
The first values of $g_n$ have been included in Table~\ref{TableNumbers}.

\begin{proof}
We recall the main steps of the proofs of Theorem~\ref{thm:GF_Unrooted_1} given in~\cite{SempleSteel2006}. To prepare the ground for future proofs, we emphasize their embedding in the context we presented in Section~\ref{sec:GF}.

\smallskip

Since counting rooted objects is far easier that counting unrooted objects, we establish a bijective correspondence between unrooted level-1 networks, and a rooted version of these networks, that we call \emph{pointed} level-1 networks. 
Pointed level-1 networks on a set of taxa $X$ are simply 
unrooted level-1 networks on the set of taxa $X \uplus \{\#\}$, 
where we declare that the leaf labeled by $\{\#\}$ is the ``root'' of the network. 
This provides a bijection between unrooted level-1 networks on the set of taxa $X \uplus \{\#\}$ and pointed level-1 networks on $X$, that have a root labeled by $\{\#\}$.
Therefore, there are as many unrooted level-1 networks on the set of taxa $X \uplus \{\#\}$ as pointed level-1 networks on $X$ rooted in a leaf labeled by $\# \notin X$. Hence $g_n$ is the number of pointed level-1 networks with $n$ leaves in addition to the root.

In a pointed level-1 network $N$ (with at least two leaves), we consider the other extremity of the edge to which the root belongs. This vertex may belong to a cycle or not. In the latter case, $N$ is simply described as an unordered pair of two pointed level-1 networks. In the former case, it is described as a non-oriented sequence of at least two pointed level-1 networks. Taking into account the trivial pointed level-1 network with one leaf, a specification for the pointed level-1 networks is therefore 
the one shown in Fig.~\ref{fig:Galled}, where an arrow labeled by \emph{sym} indicates that there is a symmetry w.r.t. the vertical axis to take into account, and the dashed edge corresponds to an edge or a path with internal vertices that are incident with cut-edges, themselves identified with edges of other pointed level-1 networks, the vertex lying on the cycle being identified with a leaf of corresponding network. 
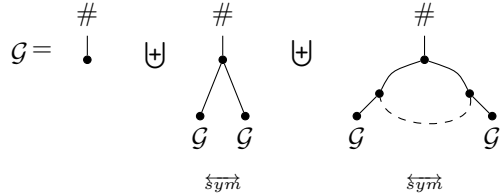
\begin{figure}[ht]
	\begin{center}
\begin{tabular}{cccc}
\begin{tikzpicture}[baseline=(top.base)]
\begin{scope}[scale=0.3]

		\draw (-6,4.5) node [] (ppointed) {\#};
		
        \draw (-3cm,2.7cm) node[inner sep=1.5pt,circle] (5) {$\biguplus$};
		\draw (-6cm,2.5cm) node[inner sep=1pt,circle,fill,draw] (top) {};
		\draw (-8cm,2.9cm) node[inner sep=1.5pt,circle] (8) {$=$};
		\draw (-9cm,2.8cm) node[inner sep=1.5pt,circle] (9) {${\mathcal G}$};
		\draw (ppointed) -- (top);

\draw (0,4.5) node [] (ppointed2) {\#};
\draw (0,2.5) node [circle, inner sep=1,fill,draw] (top2) {};
		\draw (-1cm,0cm) node[circle,inner sep=1,fill,draw] (2) {};
		\draw (1cm,0cm) node[circle,inner sep=1,fill,draw] (3) {};
		\draw (-1cm,-1cm) node[inner sep=1.5pt,circle] (4) {${\mathcal G}$};
		\draw (1cm,-1cm) node[inner sep=1.5pt,circle] (4) {${\mathcal G}$};
		\draw (0cm,-3cm) node[inner sep=1.5pt,circle] (5) {$\tiny{\overleftrightarrow{{sym}}}$};
    	\draw (top2)--(2);
		\draw (top2)--(3);  
		\draw (ppointed2) -- (top2);

\end{scope}
\end{tikzpicture}
&
$\biguplus$
&
\begin{tikzpicture}[baseline=(top.base)]
\begin{scope}[scale=0.3]
\draw (2,4.5) node [] (pointed) {\#};
\draw (2,2.5) node [circle, inner sep=1,fill,draw] (top) {};
\draw (0,1) node [circle, inner sep=1,fill,draw] (left) {};
\draw (4,1) node [circle, inner sep=1,fill,draw] (right) {};
\draw (5,-1) node {$\mathcal{G}$};
\draw (-1,0) node [circle, inner sep=1,fill,draw] (bleft) {};
\draw (-1,-1) node {$\mathcal{G}$};
\draw (5,0) node [circle, inner sep=1,fill,draw] (bright) {};
\draw (pointed) -- (top);
\draw (top) .. controls(0.5,2) .. (left);
\draw (top) .. controls(3.5,2) .. (right);
\draw  [black,dashed] (right) to[out=-80,in=-80] (left);
\draw (right) -- (bright);
\draw (left) -- (bleft);
\draw (2cm,-3cm) node[inner sep=1.5pt,circle] (16) {$\tiny{\overleftrightarrow{{sym}}}$};
\end{scope}
\end{tikzpicture}
\end{tabular}
	\end{center}
	\caption{The combinatorial specification for unrooted level-1 networks (a.k.a. galled trees).	\label{fig:Galled}}
\end{figure}

Thanks to the ``dictionary'', the generating function therefore satisfies $G(z) = z + \frac{1}{2} G(z)^2 + \frac{1}{2} \frac{G(z)^2}{1-G(z)}$ as claimed by Theorem~\ref{thm:GF_Unrooted_1}. The second statement about $G(z)$ in Theorem~\ref{thm:GF_Unrooted_1} is obtained by simple algebraic manipulations.

\smallskip

From $G(z) = z\phi(G(z))$, where $\phi(z) = \frac{1}{1-\frac{1}{2}z(1+\frac{1}{1-z})}$, we can apply Lagrange inversion to find $g_n$. Indeed, $g_n = n! [z^n] G(z) = (n-1)! [z^{n-1}] \phi(z)^n$.

Recall the following development of $(1-z)^{-n}$, for any $n\geq 1$, which will be used here and several times later on:
\begin{equation}
\left( \frac{1}{1-z} \right)^n = \sum_{i \geq 0} {{n+i-1} \choose {i}} z^i \textrm{.}
\label{eq:classical_expansion}
\end{equation}

Applying this identity twice and the binomial theorem, we get that 
\begin{align*}
\phi(z)^n = & \sum_{i \geq 0} {{n+i-1} \choose {i}} \left( \frac{1}{2}z\left(1+\frac{1}{1-z}\right)\right)^i \\
 = & \sum_{i \geq 0} {{n+i-1} \choose {i}} \left( 1+ \sum_{k=1}^i \sum_{p\geq 0} {{i} \choose {k}} {{k+p-1} \choose {p}} z^p \right) \frac{1}{2^i} z^i \\
 = & \sum_{i \geq 0} {{n+i-1} \choose {i}} \frac{z^i}{2^i} + \sum_{i \geq 0} \sum_{k=1}^i \sum_{p\geq 0} {{n+i-1} \choose {i}} {{i} \choose {k}} {{k+p-1} \choose {p}} \frac{z^{i+p}}{2^i}
\textrm{.}
\end{align*}

It follows that
\begin{align*}
[z^{n-1}] \phi(z)^n = & {{2n-2} \choose {n-1}} \frac{1}{2^{n-1}} + \sum_{i = 0}^{n-1} \sum_{k=1}^i \frac{1}{2^i} {{n+i-1} \choose {i}} {{i} \choose {k}} {{n+k-i-2} \choose {n-i-1}} \\
\textrm{ and } g_n = & \frac{(2n-2)!}{2^{n-1}(n-1)!} + \sum_{1 \leq k \leq i \leq n-1} \frac{(n+i-1)! (n+k-i-2)!}{k! (k-1)! (i-k)! (n-i-1)!}2^{-i}
\textrm{.} 
\end{align*}
\end{proof}

\subsection{Asymptotic evaluation}
From Theorem~\ref{thm:GF_Unrooted_1}, we can furthermore derive an asymptotic evaluation of the number $g_n$
of unrooted level-1 networks on $(n+1)$ leaves, using Theorem~\ref{thm:singular_inversion}.

\begin{proposition}
The number $g_n$ of unrooted level-1 networks on $(n+1)$ leaves is asymptotically equivalent to $c_1 \cdot c_2^n \cdot n^{n-1}$ for constants $c_1$ and $c_2$ such that $c_1 \approx 0.20748$ and $c_2 \approx 1.89004$.
\label{prop:asym_Unrooted_1}
\end{proposition}

\begin{proof}
Recall that $G(z)$ satisfies $G(z) = z\phi(G(z))$, where 
$\phi(z) = \frac{1}{1-\frac{1}{2}z(1+\frac{1}{1-z})}$. 
Equivalently, this can be rewritten as $\phi(z) = \frac{2-2z}{z^2 -4z+2}$. 
So, $\phi(z)$ is a rational fraction, whose pole with smallest absolute value is $2-\sqrt{2} \approx 0.5858$. 
As such, $\phi(z)$ is analytic at $0$, with radius of convergence $R=2-\sqrt{2}$. 
Moreover, owing to footnote~\ref{footnote:aperiodic}, $\phi(z)$ is aperiodic. 
Finally, the characteristic equation $\phi(z)-z\phi'(z)=0$ can be numerically solved 
(see companion Maple worksheet), showing that it admits a unique solution in the disk of convergence of $\phi$, namely $\tau \approx 0.34270$.
Therefore, the hypotheses of Theorem~\ref{thm:singular_inversion} are all satisfied, 
and denoting $\rho = \frac{\tau}{\phi(\tau)} \approx 0.19464$, Theorem~\ref{thm:singular_inversion} gives: 
$$[z^n]G(z) \sim \sqrt{\frac{\phi(\tau)}{2\phi''(\tau)}} \frac{\rho^{-n}}{\sqrt{\pi n^3}}\textrm{.}
$$

Using the Stirling estimate of the factorial $n! \sim \left(\frac{n}{e}\right)^n \sqrt{2\pi n}$, we get:
$$g_n \sim \left(\frac{n}{e}\right)^n \sqrt{2 \pi n} \sqrt{\frac{\phi(\tau)}{2\phi''(\tau)}} \frac{\rho^{-n}}{\sqrt{\pi n^3}} \sim \frac{n^{n-1}}{(e\rho)^n}\sqrt{\frac{\phi(\tau)}{\phi''(\tau)}}\textrm{.}
$$

Replacing $\tau$ and $\rho$ by their numerical approximations, we obtain the announced result. 
\end{proof}

\subsection{Refined enumeration and asymptotic distribution of parameters}

From the specification of pointed level-$1$ networks seen in the proof of Theorem~\ref{thm:GF_Unrooted_1}, 
it follows easily, as done in~\cite{SempleSteel2006}, 
that the multivariate generating function $G(z,x,y) = \sum_{n,k,m}\tfrac{g_{n,k,m}}{n!}z^n x^k y^m $,
where $g_{n,k,m}$ is the number of unrooted level-$1$ networks with $n+1$ leaves, $k$ cycles, and $m$ inner edges,
satisfies 
\begin{equation}
G(z,x,y) = z + \frac12 G(z,x,y)^2 + \frac12 x y^3 \frac{G(z,x,y)^2}{1-y G(z,x,y)}.
\label{eq:unrooted1multivariate}
\end{equation}
This equation can be rewritten as 
$G(z,x,y) = z \phi(G(z,x,y),x,y)$
where $\phi$ is defined by $\phi(z,x,y) = \tfrac{1}{1-\frac12 z \left(1+ \frac{xy^3}{1-yz}\right)}$. As done in~\cite{SempleSteel2006}, we can apply the Lagrange inversion formula to obtain an explicit expression for $g_{n,k,m}$ -- see \cite[Thm. 4]{SempleSteel2006}. 

Using Theorem~\ref{Drmota}, 
the above equation may also be used to prove that the parameters ``number of cycles'' and ``number of inner edges'' are both asymptotically normally distributed. 

\begin{proposition}
Let $X_n$ (resp. $Y_n$) be the random variable counting the number of cycles (resp. inner edges) 
in unrooted level-$1$ networks with $n+1$ leaves. 
Both $X_n$ and $Y_n$ are asymptotically normally distributed, 
and more precisely, we have 
\[
\mathbb{E} {X_n}= \mu_X n+ O(1), \;\;\;\;\;\; \mathbb{V}ar{X_n}=\sigma_X^{2} n+O(1) \;\;\;\; and \;\;\;\;\; \dfrac{X_{n}-\mathbb{E} {X_n}}{\sqrt{\mathbb{V}ar{X_n}}}\xrightarrow{d} \mathcal{N}(0,1),
\]
\[
\mathbb{E} {Y_n}= \mu_Y n+ O(1), \;\;\;\;\;\; \mathbb{V}ar{Y_n}=\sigma_Y^{2} n+O(1) \;\;\;\; and \;\;\;\;\; \dfrac{Y_{n}-\mathbb{E} {Y_n}}{\sqrt{\mathbb{V}ar{Y_n}}}\xrightarrow{d} \mathcal{N}(0,1),
\]
where $\mu_X \approx 0.46 $, $\sigma_X^2 \approx 0.18 $, $\mu_Y \approx 1.61 $ and $\sigma_Y^2 \approx 1.44 $. 
\label{prop:normaldistribunrooted1}
\end{proposition}

\begin{proof}
Consider first $X_n$. 
Defining $G(z,x) := G(z,x,1)$, it holds that 
$$
\mathbb{E} x^{X_n}= \dfrac{[z^n]G(z,x)}{[z^n]G(z,1)}.
$$
It follows from the equation for $G(z,x,y)$ that $G(z,x) = F(G(z,x),z,x)$, 
where $F$ is defined by $F(G,z,x) = z \frac{1}{1-\frac{1}{2}G\left(1+\frac{x}{1-G}\right)}$. 
Being rational, we see immediately that $F(G,z,x)$ is analytic in $G$, $z$ and $x$ around $0$. 
Moreover, performing the substitution $z=0$ (resp. $G=0$) gives $F(G,0,x) = 0$ (resp. $F(0,z,x) = z$, which is not identically $0$). 
Finally, it is readily checked that $F$ satisfies $[z^nG^m]F(G,z,1) \geq 0$ for all $n,m$
(noting for instance that $F$ is obtained using several times the quasi-inverse operator $A \mapsto \frac{1}{1-A}$, which has a combinatorial counterpart, as seen in Proposition~\ref{dictionary}).
In addition, we can determine numerically that
the system 
\begin{align*}
G=&F(G,z,1) \\
1=&F_G(G,z,1)
\end{align*}
admits a solution $(G_0,z_0)$ such that $G_0 \approx 0.3427$ and $z_0 \approx 0.1946$, 
which satisfies the hypothesis of Theorem~\ref{Drmota} 
(see the companion Maple worksheet to determine the solution and to check it satisfies the required hypotheses).
The result then follows from Theorem~\ref{Drmota}, 
and the numerical estimates of $\mu_X$ and $\sigma_X^2$ 
are obtained plugging the numerical estimates for $G_0$ and $z_0$ into the explicit formulas given  by Theorem~\ref{Drmota} (see again companion Maple worksheet for details). The proof for $Y_n$ follows the exact same steps, considering this time $G(z,y) := G(z,1,y)$ instead, and adjusting the definition of $F$ accordingly. 
		As expected, the solution $(G_0,z_0)$ of the associated system is the same as above.
\end{proof}

\begin{remark}
\label{rk:usingThm}
In the above proof of Proposition~\ref{prop:asym_Unrooted_1} (resp. Proposition~\ref{prop:normaldistribunrooted1}), 
we have provided some details on how Theorem~\ref{thm:singular_inversion} (resp. Theorem~\ref{Drmota}) was used and on how its hypotheses were checked. 
This is omitted in later proofs using Theorem~\ref{thm:singular_inversion} (see Propositions~\ref{prop:asym_Rooted_1}, \ref{prop:asym_Unrooted_2} and \ref{asy_rooted2}) 
or Theorem~\ref{Drmota} (see Propositions~\ref{prop:normaldistribunrooted1}, \ref{prop:normaldistribrooted2} and \ref{prop:normaldistribunrooted2}),  
since they work following the exact same steps. 
Note also that all numerical resolutions of equations are done in the companion Maple worksheet.

\end{remark}

\section{Counting rooted level-1 networks}\label{sec:rooted1}

\subsection{Combinatorial specification and generating function}

As for unrooted level-1 networks, we start by a combinatorial specification that describes rooted level-1 networks (also called rooted galled trees). Because every cycle in a rooted level-1 network not only has a tree vertex above all other vertices of the cycle, but also a reticulation vertex which is below all other  vertices of the cycle, notice that these objects are different from the pointed level-1 networks that we considered in the proof of Theorem~\ref{thm:GF_Unrooted_1}.

Recall that each cycle of a level-1 network has strictly more than one outgoing arc (otherwise there would be an infinite number of level-1 networks on $n$ taxa).

Let us denote by $\mathcal{R}$ the set of rooted level-1 networks. The size of a network of $\mathcal{R}$ is the number of its leaves. 
Distinguishing on the level (0 or 1) of the bridgeless component containing its root, 
a network of $\mathcal{R}$ is described in exactly one of the following ways. It may be:
\begin{itemize}
 \item a single leaf (case $0a$);
 \item a binary root vertex with two children that are roots of networks of $\mathcal{R}$, whose left-to-right order is irrelevant (case $0b$);
 \item a cycle containing the root with at least two outgoing cut arcs leading to networks of $\mathcal{R}$. This last possibility splits into two subcases, since the reticulation vertex of the cycle may be a child of the root:
\begin{itemize}
\item a cycle whose reticulation vertex is attached to a network of $\mathcal{R}$, is a child of the root and is the lowest vertex of a path coming from the root, where a sequence of at least one network of $\mathcal{R}$ is attached (case $1a$);
\item a cycle whose reticulation vertex is attached to a network of $\mathcal{R}$, and such that a sequence of at least one network of $\mathcal{R}$ is attached to each path of this cycle, the left-to-right order of these two paths being irrelevant (case $1b$).
\end{itemize}
\end{itemize}

The specification for $\mathcal{R}$ is therefore the one given in Fig.~\ref{rootedlevel1}. 
\begin{figure}[h]
	\begin{center}
		\begin{tikzpicture}[
		scale=0.7,
		level/.style={thick},
		virtual/.style={thick,densely dashed},
		trans/.style={thick,<->,shorten >=2pt,shorten <=2pt,>=stealth},
		classical/.style={thin,double,<->,shorten >=4pt,shorten <=4pt,>=stealth}
		]
		\draw (-4cm,0cm) node[circle,inner sep=1,fill,draw] (1) {};
		\draw (-4.5cm,-1cm) node[circle,inner sep=1,fill,draw] (2) {};
		\draw (-3.5cm,-1cm) node[circle,inner sep=1,fill,draw] (3) {};
		\draw (-4.5cm,-1.5cm) node[inner sep=1.5pt,circle] (4) {${\mathcal R}$};
		\draw (-3.5cm,-1.5cm) node[inner sep=1.5pt,circle] (4) {${\mathcal R}$};
		\draw (-4cm,-2cm) node[inner sep=1.5pt,circle] (5) {$\tiny{\overleftrightarrow{{sym}}}$};
    	\draw (1)--(2);
		\draw (1)--(3);

		\draw (-2cm,0cm) node[circle,inner sep=1,fill,draw] (1) {};
		\draw (-2cm,-1.5cm) node[circle,inner sep=1,fill,draw] (2) {};
		\draw (-2cm,-2cm) node[circle,inner sep=1,fill,draw] (3) {};
		\draw (-2cm,-2.5cm) node[inner sep=1.5pt,circle] (4) {${\mathcal R}$};
		\draw  [line width=2pt,blue,->] (1) to[out=0,in=45] (2);
		\draw  [black,-] (1) to[out=180,in=180] (2);
		\draw (2)--(3);

		\draw (0cm,0cm) node[circle,inner sep=1,fill,draw] (1) {};
		\draw (0cm,-1.5cm) node[circle,inner sep=1,fill,draw] (2) {};
		\draw (0cm,-2cm) node[circle,inner sep=1,fill,draw] (3) {};
		\draw (0cm,-2.5cm) node[inner sep=1.5pt,circle] (4) {${\mathcal R}$};
		\draw  [line width=2pt,blue,->] (1) to[out=0,in=45] (2);
		\draw  [line width=2pt,red,->] (1) to[out=180,in=135] (2);
		\draw (0cm,-.8cm) node[inner sep=1.5pt,circle] (5) {$\tiny{\overleftrightarrow{{sym}}}$};
		\draw (2)--(3);
				
        \draw (-1cm,-.5cm) node[inner sep=1.5pt,circle] (4) {$\biguplus$};
		\draw (-3cm,-.5cm) node[inner sep=1.5pt,circle] (5) {$\biguplus$};
		\draw (-5cm,-.5cm) node[inner sep=1.5pt,circle] (6) {$\biguplus$};
		\draw (-6cm,-.5cm) node[inner sep=1.5pt,circle] (7) {$\bullet$};
		\draw (-8cm,-.5cm) node[inner sep=1.5pt,circle] (8) {$=$};
		\draw (-9cm,-.5cm) node[inner sep=1.5pt,circle] (9) {${\mathcal R}$};
		
			\draw (-8cm,.5cm) node[inner sep=1.5pt,circle] (7) {Cases:};
			\draw (-6cm,.5cm) node[inner sep=1.5pt,circle] (7) {$0a$};
		   \draw (-4cm,.5cm) node[inner sep=1.5pt,circle] (7) {$0b$};
		   \draw (-2cm,.5cm) node[inner sep=1.5pt,circle] (7) {$1a$};
		   \draw (0cm,.5cm) node[inner sep=1.5pt,circle] (7) {$1b$};

		\end{tikzpicture}
	\end{center}		
	\caption{The combinatorial specification for rooted level-1 networks. 
	(In this picture, all arcs are directed downwards, the thick arcs each represent a directed path which contains at least one internal vertex incident with a cut arc.)}
	\label{rootedlevel1}
\end{figure}
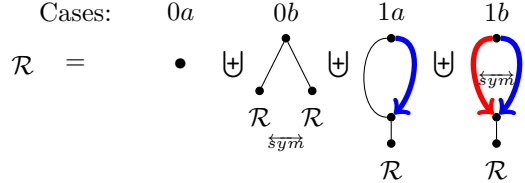

Denoting $r_n$ the number of rooted level-1 networks on $n$ leaves, and $R(z)=\sum_{n \geq 0} r_n\frac{z^n}{n!}$ the associated exponential generating function, we deduce from the specification that 
$$R=z+\frac{1}{2}R^2+\frac{R^2}{1-R}+\frac{R}{2}\left(\frac{R}{1-R}\right)^2\textrm{.}$$

Unlike for the other generating functions considered in this paper, 
the above equation for $R$ allows to find a closed formula for $R$. 
Indeed, the above equation has four solutions that can be made explicit with the help of a solver. We can further notice that evaluating the generating function $R(z)$ at $z=0$, we must obtain $R(0) = r_0 = 0$. Among the four candidate solutions for $R$, we therefore select the only one which has value 0 for $z=0$ and obtain an explicit form for $R(z)$, given in Theorem~\ref{thm:GF_Rooted_1}.

\begin{theorem}\label{thm:GF_Rooted_1}
The exponential generating function $R(z)$ of rooted level-1 networks is expressed as 
$$R(z)=-\frac{\sqrt 2 \sqrt{-4(\sqrt{1-8z} - 2)z + 9 \sqrt{1-8z} - 1} }{4(1-8z)^{\frac{1}{4}}} - \frac{1}{4} \sqrt{1-8z} + \frac{5}{4}$$
within its disk of convergence of radius $\frac{1}{8}$.
\end{theorem}

\subsection{Exact enumeration formula}
The first terms of the sequence $(r_0,r_1, r_2, \ldots)$ can be read on the Taylor expansion of $R(z)$, and have been collected in Table~\ref{TableNumbers}:
$$R(z) = z + 3 \frac{z^2}{2!} + 36 \frac{z^3}{3!} + 723 \frac{z^4}{4!} + 20280 \frac{z^5}{5!} + o(z^{5}) \textrm{.}$$

More generally, we have:
\begin{proposition}
For any $n \geq 1$, the number $r_n$ of rooted level-1 networks with $n$ leaves is given by 
$$\frac{(2n-2)!}{2^{n-1}(n-1)!} + \sum_{1 \leq k \leq i \leq n-1 \atop 0 \leq p\leq k} \frac{(n+i-1)! (n+k-i-2)! \quad 2^{p-i}}{(i-k)!(k-p)!p!(n-1-i-k+p)!(2k-p-1)!}\textrm{.}
$$
\label{prop:enum_Rooted_1}
\end{proposition}

\begin{proof}
To obtain a generic formula for $r_n$, we apply the Lagrange inversion formula, rewriting $R(z)$ as $R(z)=z \phi(R(z))\textrm{ where }\phi (z) = \frac{1}{1-\frac{1}{2}z-\frac{z}{1-z}-\frac{1}{2}\left(\frac{z}{1-z}\right)^2}\textrm{.}$

Using twice the usual development of $(1-z)^{-n}$ (for $n \geq 1$) which we recalled in Eq.~\eqref{eq:classical_expansion} 
and twice the binomial theorem, we obtain that
\begin{align*}
\phi(z)^n = &\sum_{i \geq 0} {{n+i-1} \choose {i}} \frac{z^i}{2^i} \\
+ & \sum_{i \geq 0} \sum_{k=1}^i \sum_{p=0}^k \sum_{j \geq 0} {{n+i-1} \choose {i}} {{i} \choose {k}} {{k} \choose {p}} {{2k-p+j-1} \choose {j}} \frac{z^{i+k-p+j}}{2^{i-p}}\textrm{,}
\end{align*}

and we deduce that
\begin{align*}
r_n & = n! [z^n] R(z) = n! \frac{1}{n} [z^{n-1}] \phi(z)^n = (n-1)! [z^{n-1}] \phi(z)^n \\
 & = \frac{(2n-2)!}{2^{n-1}(n-1)!} \\
 & + \sum_{1 \leq k \leq i \leq n-1 \atop 0 \leq p\leq k} \frac{(n+i-1)! (n+k-i-2)!}{(i-k)!(k-p)!p!(n-1-i-k+p)!(2k-p-1)!} 2^{p-i}
\end{align*}

as announced.
\end{proof}

\subsection{Asymptotic evaluation}

The equation for $R(z)$ also enables us to derive an asymptotic estimate of $r_n$. 

\begin{proposition}
The number $r_n$ of rooted level-1 networks on $n$ leaves is asymptotically equivalent to $c_1 \cdot c_2^n \cdot n^{n-1}$ for $c_1 = \frac{\sqrt{34}(\sqrt{17}-1)}{136} \approx 0.1339$ and $c_2 = \frac{8}{e} \approx 2.943$.
\label{prop:asym_Rooted_1}
\end{proposition}

\begin{proof}
Recall that  $R(z)=z \phi(R(z))\textrm{ where }\phi (z) = \frac{1}{1-\frac{1}{2}z-\frac{z}{1-z}-\frac{1}{2}\left(\frac{z}{1-z}\right)^2}$ so that we can apply the Singular Inversion Theorem. Unlike in the case of unrooted level-1 networks, the solution $\tau$ of the characteristic equation $\phi(z)-z\phi'(z)=0$ to be considered has a nice explicit expression here, and we have $\tau =\frac{5- \sqrt{17}}{4}$. We obtain $\rho = \frac{\tau}{\phi(\tau)} = \frac{1}{8}$ and $\sqrt{\frac{\phi(\tau)}{2\phi''(\tau)}} = \frac{\sqrt{17} (\sqrt{17}-1)}{136}$. 
Consequently, from Theorem~\ref{thm:singular_inversion} we have:
$$ [z^n] R(z) \sim \frac{\sqrt{17} (\sqrt{17}-1)}{136} \frac{8^n}{\sqrt{\pi n^3}}\textrm{.}
$$

Since $r_n = n! [z^n] R(z)$, using the Stirling estimate of the factorial, we finally get:
$$ r_n \sim \frac{\sqrt{34}(\sqrt{17}-1)}{136} \left(\frac{8}{e}\right)^n n^{n-1}\textrm{.} \hfill
$$
\end{proof}

Notice that with the explicit expression of the generating function $R(z)$ in Theorem~\ref{thm:GF_Rooted_1}, another way of proving Proposition~\ref{prop:asym_Rooted_1} would have been to use the Transfer Theorem (Corollary VI.1 of~\cite{FlajoletSedgewick2008}). We do not enter the details of this other method here, but we can check that it gives the same result.

\subsection{Refined enumeration formula}

As in the work of Semple and Steel~\cite{SempleSteel2006}, we can refine the enumeration of rooted level-1 networks according to two additional parameters, which are typical of the ``level-1'' nature of our networks: their number of cycles and their total number of arcs among cycles. To do so, let us introduce the multivariate generating function $R(z,x,y) = \sum \frac{r(n,k,m)}{n!} z^n x^k y^m$, where $r(n,k,m)$ is the number of rooted level-1 networks with $n$ leaves, $k$ cycles and $m$ inner arcs (i.e. the total number of arcs inside those $k$ cycles is $m$). The specification for $\mathcal{R}$ translates into the following equation for $R= R(z,x,y)$:
\begin{equation}
R=z+\frac{1}{2}R^2+x\frac{R^2y^3}{1-yR}+xR\frac{1}{2}\left(\frac{Ry^2}{1-yR}\right)^2\textrm{.}
\label{eq:rooted1multivariate}
\end{equation}

The equation can be rewritten as follows:
$$R=z \phi(R,x,y)\textrm{ where }\phi (z,x,y) = \frac{1}{1-\frac{1}{2}z-x\frac{zy^3}{1-yz}-\frac{1}{2}xy^4\left(\frac{z}{1-yz}\right)^2}\textrm{.}$$

Applying the Lagrange inversion formula again, we have 
$$\frac{r(n,k,m)}{n!} = [z^n x^k y^m] R(z,x,y) = \frac{1}{n}[z^{n-1} x^k y^m] \phi(z,x,y)^n,$$ 
and by the exact same steps of computation as in the proof of Proposition~\ref{prop:enum_Rooted_1}, we get:

\begin{proposition}
The number $r(n,k,m)$ of level-1 networks with $n$ leaves, $k$ cycles and $m$ inner arcs 
(with $k\geq 1$ and $m\geq 1$) is 
$$
r(n,k,m) = \sum_{p=0}^k \frac{(2n+3k-m-2)! (m-2k-1)! 2^{p+m+1-n-3k}}{(n+2k-m-1)! p! (k-p)! (m-4k+p)! (2k-p-1)!} \textrm{.}
$$
\label{prop:refinedEnumerationRooted1}
\end{proposition}

Notice that from $r_n = r(n,0,0) + \sum_{k=1}^{n-1} \sum_{m=3k}^{n+2k-1} r(n,k,m)$ and the above theorem, we can recover Proposition~\ref{prop:enum_Rooted_1} by the change of variable $m=n+3k-i-1$.

\subsection{Asymptotic distribution of parameters}

As we have seen with Proposition~\ref{prop:normaldistribunrooted1}, 
the equation for the refined generating function does not only 
give access to the explicit formula of Proposition~\ref{prop:refinedEnumerationRooted1} above, 
but also allows to prove that the two parameters of interest 
are each asymptotically normally distributed. 

\begin{proposition}\label{prop:normaldistribrooted2}
Let $X_n$ (resp. $Y_n$) be the random variable counting the number of cycles (resp. inner arcs) 
in rooted level-$1$ networks with $n$ leaves. 
Both $X_n$ and $Y_n$ are asymptotically normally distributed, 
and more precisely, we have 
\[
\mathbb{E} {X_n}= \mu_X n+ O(1), \;\;\;\;\;\; \mathbb{V}ar{X_n}=\sigma_X^{2} n+O(1) \;\;\;\; and \;\;\;\;\; \dfrac{X_{n}-\mathbb{E} {X_n}}{\sqrt{\mathbb{V}ar{X_n}}}\xrightarrow{d} \mathcal{N}(0,1),
\]
\[
\mathbb{E} {Y_n}= \mu_Y n+ O(1), \;\;\;\;\;\; \mathbb{V}ar{Y_n}=\sigma_Y^{2} n+O(1) \;\;\;\; and \;\;\;\;\; \dfrac{Y_{n}-\mathbb{E} {Y_n}}{\sqrt{\mathbb{V}ar{Y_n}}}\xrightarrow{d} \mathcal{N}(0,1),
\]
where $\mu_X \approx 0.56 $, $\sigma_X^2 \approx 0.18 $, $\mu_Y \approx 1.93 $ and $\sigma_Y^2 \approx 1.24 $. 
\end{proposition}

\begin{proof}
Recall that, defining $\phi (z,x,y) = \frac{1}{1-\frac{1}{2}z-x\frac{zy^3}{1-yz}-\frac{1}{2}xy^4\left(\frac{z}{1-yz}\right)^2}$, $R(z,x,y)$ satisfies $R=z \phi(R,x,y)$. 
We focus first on $X_n$, setting $y=1$, and we consider $R(z,x) := R(z,x,1)$. 
It holds that 
$$
\mathbb{E} x^{X_n}= \dfrac{[z^n]R(z,x)}{[z^n]R(z,1)}.
$$
Defining the function $F$ by $F(R,z,x) = z\phi(R,x,1)$, 
it also holds that $R(z,x) = F(R(z,x),z,x)$. 
It is readily checked that $F$ satisfies all hypotheses of Theorem~\ref{Drmota}. 
Moreover, the system 
\begin{align*}
R=&F(R,z,1) \\
1=&F_R(R,z,1)
\end{align*}
admits a solution $(R_0,z_0)$ with $z_0= 1/8$ and $R_0 \approx 0.2192$, 
which satisfies the hypothesis of Theorem~\ref{Drmota}.
The result and numerical estimates of $\mu_X$ and $\sigma_X^2$ then follow from Theorem~\ref{Drmota}.

For $Y_n$ instead of $X_n$, 
the proof works in the exact same way, considering this time $R(z,y) := R(z,1,y)$ instead, and adjusting the definition of $F$ accordingly. 
As in the proof of Proposition~\ref{prop:normaldistribunrooted1}, 
we find the same solution $(R_0,z_0)$ of the associated system, as it should be.
\end{proof}

\section{Counting unrooted level-2 networks} \label{sec:unrooted2}

\subsection{Combinatorial specification}

First of all, let us recall that any bridgeless component in an unrooted level-2 network contains at least three vertices incident with a cut-edge (since otherwise there would be an infinite number of such networks with a given number of leaves). 

As in the case of level-1 unrooted networks, we consider \emph{pointed} level-2 networks, that are unrooted level-2 networks equipped with a fictitious root, which is a new leaf labeled by the special taxa $\#$. This provides a bijection between unrooted level-2 networks on the set of taxa $X \uplus \{\#\}$ and pointed level-2 networks on $X$.
Therefore, there are as many unrooted level-2 networks of the set of taxa $X \uplus \{\#\}$ as pointed level-2 networks on $X$ rooted in a leaf labeled by $\# \notin X$.
Notice that pointed level-2 networks do not correspond to classical rooted level-2 networks. Indeed, every bridgeless component in a pointed level-2 network has a distinguished vertex which could be considered as the equivalent of a root, but no reticulation vertices, whereas it has both in the usual definition of rooted level-2 networks.

Let us denote by $\mathcal{U}$ the set of such pointed level-2 networks, the size of a network of $\mathcal{U}$ being the number of its leaves different from the root. 
Let $u_n$ be the number of networks of size $n$ in $\mathcal{U}$. 
The above argument shows that $u_n$ counts the number of unrooted level-2 networks on $(n+1)$ leaves. We introduce $U(z)=\sum_{n \geq 0} u_n\frac{z^n}{n!}$ the associated exponential generating function.

To obtain a combinatorial specification for $\mathcal{U}$, and hence an equation satisfied by $U(z)$, we describe the possible shapes of a network $N$ of $\mathcal{U}$, depending on 
the level (0, 1 or 2) of the bridgeless component that contains the neighbouring vertex of the fictitious root. 

Let $v$ be the neighbor of the fictitious root. Then we have the following cases to consider, depending on the level of the bridgeless component that contains $v$. 

\begin{itemize}
 \item The first case is that $v$ is a leaf.
 \item If $v$ does not belong to a cycle nor to a bridgeless component of level 2, then $N$ is described as an unordered pair of two pointed level-2 networks.
 \item If $v$ belongs to a cycle but not to a bridgeless component of level 2, then $N$ is described as an unoriented sequence of at least two pointed level-2 networks. \\
 (These first three cases are the same as in Section~\ref{sec:rooted1}.)
 \item The last possibility is that $v$ belongs to a bridgeless component of level 2. 
 The underlying level-2 generator, $G$, is necessarily of the shape 
\begin{tikzpicture}[baseline=(top.base)]
\begin{scope}[scale=0.3]
\draw (2,4.5) node [] (pointed) {\#};
\draw (2,2.5) node [circle, inner sep=1,fill,draw] (top) {};
\draw (0,1) node [circle, inner sep=1,fill,draw] (left) {};
\draw (4,1) node [circle, inner sep=1,fill,draw] (right) {};
\draw (pointed) -- (top);
\draw (top) .. controls(0.5,2) .. (left);
\draw (top) .. controls(3.5,2) .. (right);
\draw (right) .. controls(3,-1) .. (left);
\draw (left) .. controls(1,-1) .. (right);
\end{scope}
\end{tikzpicture}.
\noindent In this case, we distinguish many cases in Section~\ref{appendix-unrooted-l2} of the Appendix, depending on whether each edge of the level-2 generator contains exactly one vertex incident with a cut-edge, several or none.
\end{itemize}

\subsection{Generating function}
The specification is directly translated into the following equation for the generating function $U$ (see the Appendix for details):
\begin{align*}\label{eq:unrooted2multivariate}
{U}&=z+
\frac{{U}^2}{2}+\frac{{U}^2}{2(1- {U})} 
+\frac{{U}^2}{2(1- {U})} 
+\frac{3}{2}{U}^2 
+\frac{5{U}^3}{2(1- {U})} 
+\frac{5{U}^4}{4(1- {U})^2} 
+{U}^3 
+\frac{3{U}^4}{1-{U}} 
\\
 	&
+\frac{3{U}^5}{(1- {U})^2} 
+\frac{{U}^6}{(1- {U})^3} 
+\frac{{U}^4}{4} 
+\frac{{U}^5}{1- {U}} 
+\frac{3{U}^6}{2(1- {U})^2} 
+\frac{{U}^7}{(1- {U})^3} 
+\frac{{U}^8}{4(1- {U})^4}. 
\end{align*}

This equation for the generating function allows to derive the first coefficients of the series expansion of $U(z)$, namely: 
 \[
 {U}(z)=z+3z^2+\frac{45}{2}z^3+\frac{421}{2}z^4+\frac{8809}{4}z^5+\cdots.
 \]
The corresponding first values of $u_n$ have been included in Table~\ref{TableNumbers}. 
(Recall indeed that $U(z)=\sum_{n \geq 0} u_n\frac{z^n}{n!}$ and that $u_n$ is the number of unrooted level-2 networks on $(n+1)$ leaves). 

The above equation for $U(z)$ can also be rewritten as follows: 
\begin{theorem}
The generating function $U(z)$ satisfies:  
$$U(z)=z \phi(U(z))\textrm{ where }\phi (z) = \frac{1}{1-\frac{3z^5-16z^4+32z^3-30z^2+12z}{4(1-z)^4}}\textrm{.}$$
\label{thm:GF_Unrooted_2}
\end{theorem}
\begin{proof}
This is simply obtained from the above equation for $U$ by algebraic manipulations. 
\end{proof}

\subsection{Exact enumeration formula}
To obtain a closed form for $u_n$, we start from the equation for $U$ given in Theorem~\ref{thm:GF_Unrooted_2}.
By the Lagrange inversion formula we obtain that:
$$u_n = n! [z^n]U(z) = \frac{n!}{n}[z^{n-1}] \phi^n(z) = (n-1)! [z^{n-1}] \phi^n(z),$$
so, to compute the first values of $u_n$, we can compute
the Taylor expansions of $\phi^n(z)$ and get the values in Table~\ref{TableNumbers}. 

As for the case of level-1 networks, we may also deduce with routine algebra 
an explicit formula for $u_n$. 
This formula being however rather involved, we provide it only in Appendix.

\subsection{Asymptotic evaluation}

From Theorem~\ref{thm:GF_Unrooted_2}, we can furthermore derive an asymptotic evaluation of the number $u_n$ of unrooted level-2 networks on $(n+1)$ leaves, using Theorem~\ref{thm:singular_inversion}.

\begin{proposition}
The number $u_n$ of unrooted level-2 networks on $(n+1)$ leaves is asymptotically equivalent to $c_1 \cdot c_2^n \cdot n^{n-1}$ for constants $c_1$ and $c_2$ such that $c_1 \approx 0.07695$ and $c_2 \approx 5.4925$.
\label{prop:asym_Unrooted_2}
\end{proposition}

\begin{proof}
Denoting by $\tau \approx 0.12117$ the unique solution of the characteristic equation $\phi(z)-z\phi'(z)=0$ in the disk of convergence of $\phi$, and $\rho = \frac{\tau}{\phi(\tau)} \approx 0.06698$, we have:
$$[z^n]U(z) \sim \sqrt{\frac{\phi(\tau)}{2\phi''(\tau)}} \frac{\rho^{-n}}{\sqrt{\pi n^3}}\textrm{.}
$$

Using the Stirling estimate of the factorial, we get:
$$u_n \sim \left(\frac{n}{e}\right)^n \sqrt{2 \pi n} \sqrt{\frac{\phi(\tau)}{2\phi''(\tau)}} \frac{\rho^{-n}}{\sqrt{\pi n^3}} \sim \frac{n^{n-1}}{(e\rho)^n}\sqrt{\frac{\phi(\tau)}{\phi''(\tau)}}\textrm{.}
$$

Replacing $\tau$ and $\rho$ by their numerical approximations, we get the announced result. 
\end{proof}

\subsection{Refined enumeration formula and asymptotic distribution of parameters}
 Consider the refined generating function $ {U}(z,x,y)$ 
 for unrooted level-2 networks, 
 where the variable $z$ counts the size as before, 
 the variable $x$ counts the number of bridgeless components of level 1 or 2 
 (or equivalently, the number of level-1 or level-2 generators in the decomposition of these networks), 
 and the variable $y$ counts the number of inner edges, defined as the total number of edges across all level-1 and level-2 bridgeless components. 
 The specification provided in the Appendix can be refined for these statistics, 
 yielding the following equation for $ U := {U}(z,x,y)$:  

\begin{align}
 	{U}=z&+\frac{{U}^2}{2}+\frac{xy^3{U}^2}{2(1-y {U})}+\frac{xy^6{U}^2}{2(1-y {U})}+\frac{3}{2}xy^6 {U}^2+\frac{5xy^7{U}^3}{2(1-y {U})}+\frac{5xy^8{U}^4}{4(1-y {U})^2} \nonumber \\
 	&+xy^7{U}^3+\frac{3xy^{8}{U}^4}{1-y {U}}+\frac{3xy^{9}{U}^5}{(1-y {U})^2}+\frac{xy^{10}{U}^6}{(1-y {U})^3}+\frac{xy^8{U}^4}{4}+\frac{xy^{9}{U}^5}{1-y {U}} \nonumber\\
 	&+\frac{3xy^{10}{U}^6}{2(1-y {U})^2}+\frac{xy^{11}{U}^7}{(1-y {U})^3}+\frac{xy^{12}{U}^8}{4(1-y {U})^4}.
  \label{eq:multi_unrooted2}
 \end{align}

From the above equation, and similarly to Proposition~\ref{prop:refinedEnumerationRooted1}, it would be possible (although computations and result are not reported in this paper) to derive an explicit formula for the number of unrooted level-2 networks with $n$ leaves, $k$ bridgeless components of level 1 or 2, and $m$ edges across them. 
Furthermore, some information on the asymptotic behavior of these parameters can be obtained from  Eq.~\eqref{eq:multi_unrooted2}. 

\begin{proposition}
Let $X_n$ (resp. $Y_n$) be the random variable counting the number of 
level-1 or level-2 bridgeless components (resp. the number of edges across them)
in unrooted level-$2$ networks with $n+1$ leaves. 
Both $X_n$ and $Y_n$ are asymptotically normally distributed, 
and more precisely, we have 
\[
\mathbb{E} {X_n}= \mu_X n+ O(1), \;\;\;\;\;\; \mathbb{V}ar{X_n}=\sigma_X^{2} n+O(1) \;\;\;\; and \;\;\;\;\; \dfrac{X_{n}-\mathbb{E} {X_n}}{\sqrt{\mathbb{V}ar{X_n}}}\xrightarrow{d} \mathcal{N}(0,1),
\]
\[
\mathbb{E} {Y_n}= \mu_Y n+ O(1), \;\;\;\;\;\; \mathbb{V}ar{Y_n}=\sigma_Y^{2} n+O(1) \;\;\;\; and \;\;\;\;\; \dfrac{Y_{n}-\mathbb{E} {Y_n}}{\sqrt{\mathbb{V}ar{Y_n}}}\xrightarrow{d} \mathcal{N}(0,1),
\]
where $\mu_X \approx 0.69944 $, $\sigma_X^2 \approx 0.16919 $, $\mu_Y \approx 4.01349 $ and $\sigma_Y^2 \approx 4.68675 $. 
\label{prop:normaldistribunrooted2}
\end{proposition}

\begin{proof}
Consider first $X_n$. 
Defining $U(z,x) := U(z,x,1)$, it holds that 
$$
\mathbb{E} x^{X_n}= \dfrac{[z^n]U(z,x)}{[z^n]U(z,1)}.
$$
It follows from the equation for $U(z,x,y)$ that $U(z,x) = F(U(z,x),z,x)= z \frac{1}{1-A(U(z,x),z,x)}$ with 
\begin{align*}
A(U,z,x) =& 
\frac{{U}}{2}+\frac{x{U}}{2(1-{U})}+\frac{x{U}}{2(1-{U})}+\frac{3x {U}}{2}+\frac{5x{U}^2}{2(1- {U})}+\frac{5x{U}^3}{4(1- {U})^2}+x{U}^2\\ 	&+\frac{3x{U}^3}{1-{U}}+\frac{3x{U}^4}{(1-{U})^2}+\frac{x{U}^5}{(1-{U})^3}+\frac{x{U}^3}{4}+\frac{x{U}^4}{1-{U}}+\frac{3x{U}^5}{2(1-{U})^2}\\
&+\frac{x{U}^6}{(1-{U})^3}+\frac{x{U}^7}{4(1- {U})^4}.
 \end{align*}
It is readily checked that $F$ satisfies all hypotheses of Theorem~\ref{Drmota}. 
The system 
\begin{align*}
U=&F(U,z,1) \\
1=&F_U(U,z,1)
\end{align*}
admits a solution $(U_0,z_0)$ such that $U_0 \approx 0.1212$ and $z_0 \approx 0.06698$, 
which satisfies the hypothesis of Theorem~\ref{Drmota}.
The result then follows from Theorem~\ref{Drmota}, 
and the numerical estimates of $\mu_X$ and $\sigma_X^2$ 
are obtained plugging the numerical estimates for $U_0$ and $z_0$ into the explicit formulas given  by Theorem~\ref{Drmota}. The proof for $Y_n$ follows the exact same steps, considering this time $U(z,y) := U(z,1,y)$ instead, and adjusting the definition of $F$ accordingly. 
Again, as expected, the solution $(U_0,z_0)$ of the associated system is the same as above.
\end{proof}

\section{Counting rooted level-2 networks}\label{sec:rooted2}
        
\subsection{Combinatorial specification and generating function}
To derive a specification for rooted level-2 networks, 
we distinguish cases depending on the level ($0$, $1$ or $2$) of the generator to which the root belongs. 
The cases corresponding to levels $0$ and $1$ will be the same as in Section~\ref{sec:rooted1}. 
When the root of a rooted level-2 network belongs to a level-$2$ generator, 
we have to remember that these generators have one vertex which is above all their other vertices (which is the root of the network) and not just one but two reticulation vertices. 
As for rooted level-1 networks, it is important to keep in mind that 
any bridgeless component of level $2$ in a rooted level-2 network has 
at least two outgoing cut arcs 
(since otherwise there would be an infinite number of such networks with a given number of leaves).

We denote by $\mathcal{L}$ the set of rooted level-2 networks, 
where the size corresponds to the number of leaves. 
And we denote by $L(z)$ the corresponding exponential generating function. 
Distinguishing on the level (0, 1 or 2) of the bridgeless component containing the root, 
we can see that any 
network $N$ of $\mathcal{L}$ satisfies exactly one of the following (see Fig.~\ref{rootedlevel2}). 
\begin{itemize}
    \item  $N$ is just a leaf. 
    This contributes $z$ to the generating function (case $0a$). 
    \item The root of $N$ belongs to a bridgeless component of level 0, 
    that is to say it is a binary root vertex. 
    Its children are themselves networks of $\mathcal{L}$ whose left-to-right order is irrelevant. 
    This contributes $\frac{L^2}{2}$ to the generating function (case $0b$). 
    \item The root of $N$ belongs to a bridgeless component of level 1. 
    This case splits into two subcases, just as in Section~\ref{sec:rooted1}.  
    \begin{itemize}
 		\item Either $N$ consists of a cycle whose reticulation vertex is attached to a network of $\mathcal{L}$, is a child of the root and is the lowest vertex of a path from the root where a sequence of at least one network of $\mathcal{L}$ is attached. This contributes $\frac{L^2}{1-L}$ to the generating function (case $1a$). 
 		\item Or $N$ consists of a cycle whose reticulation vertex is attached to a network of $\mathcal{L}$, and a sequence of at least one network of $\mathcal{L}$ is attached to each path of this cycle (case $1b$).
 		This contributes $\frac{L}{2}\left(\frac{L}{1-L}\right)^2$. 
 	\end{itemize}  
    \item The root of $N$ belongs to a bridgeless component of level 2. 
    The level-2 generators are displayed in Fig.~\ref{rootedlevel2}, cases $2a$ to $2d$. 
    From these generators, the networks whose root belong to a bridgeless component of level 2 are obtained 
    attaching networks of $\mathcal{L}$ to their reticulation vertex or vertices with out-degree 0,
    and possibly replacing their arcs with sequences of at least one network of $\mathcal{L}$. Note that in cases $2b$ and $2d$, depending on our choices for such arcs, we may have to cope with horizontal and vertical symmetry. 
    We study these cases in order, and find their contribution to the generating function $L(z)$. 
\end{itemize}
 
\begin{figure}[h]
	\begin{center}
		\begin{tikzpicture}[
		scale=0.7,
		level/.style={thick},
		virtual/.style={thick,densely dashed},
		trans/.style={thick,<->,shorten >=2pt,shorten <=2pt,>=stealth},
		classical/.style={thin,double,<->,shorten >=4pt,shorten <=4pt,>=stealth}
		]
		\draw (-4cm,0cm) node[circle,inner sep=1,fill,draw] (1) {};
		\draw (-4.5cm,-1cm) node[circle,inner sep=1,fill,draw] (2) {};
		\draw (-3.5cm,-1cm) node[circle,inner sep=1,fill,draw] (3) {};
		\draw (-4.5cm,-1.5cm) node[inner sep=1.5pt,circle] (4) {${\mathcal L}$};
		\draw (-3.5cm,-1.5cm) node[inner sep=1.5pt,circle] (4) {${\mathcal L}$};
		\draw (-4cm,-2cm) node[inner sep=1.5pt,circle] (5) {$\tiny{\overleftrightarrow{{sym}}}$};
    	\draw (1)--(2);
		\draw (1)--(3);

		\draw (-2cm,0cm) node[circle,inner sep=1,fill,draw] (1) {};
		\draw (-2cm,-1.5cm) node[circle,inner sep=1,fill,draw] (2) {};
		\draw (-2cm,-2cm) node[circle,inner sep=1,fill,draw] (3) {};
		\draw (-2cm,-2.5cm) node[inner sep=1.5pt,circle] (4) {${\mathcal L}$};
		\draw  [blue,->] (1) to[out=0,in=0] (2);
		\draw  [black,-] (1) to[out=180,in=180] (2);
		\draw (2)--(3);

		\draw (0cm,0cm) node[circle,inner sep=1,fill,draw] (1) {};
		\draw (0cm,-1.5cm) node[circle,inner sep=1,fill,draw] (2) {};
		\draw (0cm,-2cm) node[circle,inner sep=1,fill,draw] (3) {};
		\draw (0cm,-2.5cm) node[inner sep=1.5pt,circle] (4) {${\mathcal L}$};
		\draw  [blue,->] (1) to[out=0,in=0] (2);
		\draw  [red,->] (1) to[out=180,in=180] (2);
		\draw (0cm,-.8cm) node[inner sep=1.5pt,circle] (5) {$\tiny{\overleftrightarrow{{sym}}}$};
		\draw (2)--(3);
		
		\draw (-6.5cm,-3.5cm) node[circle,inner sep=1,fill,draw] (1) {};
		\draw (-5.8cm,-4cm) node[circle,inner sep=1,fill,draw] (2) {};
		\draw (-7cm,-4.5cm) node[circle,inner sep=1,fill,draw] (3) {};
		\draw (-6.5cm,-5cm) node[circle,inner sep=1,fill,draw] (4) {};
		\draw (-6.5cm,-5.5cm) node[circle,inner sep=1,fill,draw] (5) {};
		\draw (-6.5cm,-6cm) node[inner sep=1.5pt,circle] (6) {${\mathcal L}$};
		\draw (1)--(2);\draw (1)--(3);\draw (2)--(3);\draw (2)--(4);\draw (3)--(4);\draw (4)--(5);

	\draw (-3cm,-3.5cm) node[circle,inner sep=1,fill,draw] (1) {};
	\draw (-4cm,-4cm) node[circle,inner sep=1,fill,draw] (2) {};
	\draw (-4cm,-5cm) node[circle,inner sep=1,fill,draw] (3) {};
	\draw (-3cm,-5.5cm) node[circle,inner sep=1,fill,draw] (4) {};
	\draw (-3cm,-6.2cm) node[circle,inner sep=1,fill,draw] (5) {};
	\draw (-3cm,-6.7cm) node[inner sep=1.5pt,circle] (6) {${\mathcal L}$};
	\draw (1)--(2);\draw (1)--(4);\draw (3)--(4);\draw (4)--(5);
	\draw  [black,-] (2) to[out=0,in=0] (3);
	\draw  [black,-] (2) to[out=180,in=180] (3);
	\draw (-4.6cm,-4.6cm) node[inner sep=1.5pt,circle] (7) {$e$};
	\draw (-3.3cm,-4.5cm) node[inner sep=1.5pt,circle] (8) {$e^\prime$};
	\draw (-4cm,-5.5cm) node[inner sep=1.5pt,circle] (9) {$\tiny{\overleftrightarrow{{sym}}}$};

	\draw (-.5cm,-3.5cm) node[circle,inner sep=1,fill,draw] (1) {};
	\draw (-1cm,-4cm) node[circle,inner sep=1,fill,draw] (2) {};
	\draw (.2cm,-5cm) node[circle,inner sep=1,fill,draw] (3) {};
	\draw (-.5cm,-4.5cm) node[circle,inner sep=1,fill,draw] (4) {};
	\draw (-1cm,-5cm) node[circle,inner sep=1,fill,draw] (5) {};
	\draw (-1cm,-5.5cm) node[circle,inner sep=1,fill,draw] (6) {};
	\draw (.2cm,-5.5cm) node[circle,inner sep=1,fill,draw] (7) {};
	\draw (-1cm,-6cm) node[inner sep=1.5pt,circle] (8) {${\mathcal L}$};
	\draw (.2cm,-6cm) node[inner sep=1.5pt,circle] (9) {${\mathcal L}$};
	\draw (1)--(2);\draw (1)--(3);\draw (2)--(4);\draw (2)--(5);\draw (4)--(3);\draw (4)--(5);\draw (3)--(7);\draw (5)--(6);
	
		
		\draw (2cm,-3.5cm) node[circle,inner sep=1,fill,draw] (1) {};
		\draw (1cm,-4cm) node[circle,inner sep=1,fill,draw] (2) {};
		\draw (3cm,-4cm) node[circle,inner sep=1,fill,draw] (3) {};
		\draw (2cm,-4.5cm) node[circle,inner sep=1,fill,draw] (4) {};
		\draw (2cm,-6.5cm) node[circle,inner sep=1,fill,draw] (5) {};
		\draw (2cm,-5.2cm) node[circle,inner sep=1,fill,draw] (6) {};
		\draw (2cm,-5.7cm) node[inner sep=1.5pt,circle] (8) {${\mathcal L}$};
		\draw (2cm,-7.5cm) node[circle,inner sep=1,fill,draw] (7) {};
		\draw (2cm,-8cm) node[inner sep=1.5pt,circle] (9) {${\mathcal L}$};
	    \draw (1)--(2);\draw (1)--(3);\draw (2)--(4);\draw (3)--(4);\draw (4)--(6);\draw (2)--(5);\draw (3)--(5);\draw (5)--(7);
		\draw (1.3cm,-3.6cm) node[inner sep=1.5pt,circle] (10) {$e_{1}$};
		\draw (2.7cm,-3.6cm) node[inner sep=1.5pt,circle] (11) {$e_{1}^\prime$};
		\draw (1.6cm,-4.6cm) node[inner sep=1.5pt,circle] (12) {$e_{2}$};
		\draw (2.4cm,-4.6cm) node[inner sep=1.5pt,circle] (13) {$e_{2}^\prime$};
		\draw (1.3cm,-5.5cm) node[inner sep=1.5pt,circle] (14) {$e_{3}$};
		\draw (2.7cm,-5.5cm) node[inner sep=1.5pt,circle] (15) {$e_{3}^\prime$};
		\draw (2cm,-4.1cm) node[inner sep=1.5pt,circle] (16) {$\tiny{\overleftrightarrow{{sym}}}$};
		
	\draw (2.5cm,-4.2cm) node[inner sep=1.5pt,circle] (17) {};
		\draw (2.2cm,-6cm) node[inner sep=1.5pt,circle] (18) {};
		\draw (3.6cm,-5.5cm) node[inner sep=1.5pt,circle] (19) {$ \tiny{sym} $};
		\draw  [black,<->] (17) to[out=-20,in=-20] (18);

		\draw (.5cm,-4cm) node[inner sep=1.5pt,circle] (1) {$\biguplus$};
		\draw (-2cm,-4cm) node[inner sep=1.5pt,circle] (1) {$\biguplus$};
		\draw (-5cm,-4cm) node[inner sep=1.5pt,circle] (2) {$\biguplus$};
		\draw (-8cm,-4cm) node[inner sep=1.5pt,circle] (3) {$\biguplus$};
		\draw (-1cm,-.5cm) node[inner sep=1.5pt,circle] (4) {$\biguplus$};
		\draw (-3cm,-.5cm) node[inner sep=1.5pt,circle] (5) {$\biguplus$};
		\draw (-5cm,-.5cm) node[inner sep=1.5pt,circle] (6) {$\biguplus$};
		\draw (-6cm,-.5cm) node[inner sep=1.5pt,circle] (7) {$\bullet$};
		\draw (-8cm,-.5cm) node[inner sep=1.5pt,circle] (8) {$=$};
		\draw (-9cm,-.5cm) node[inner sep=1.5pt,circle] (9) {${\mathcal L}$};
		
			\draw (-6cm,.5cm) node[inner sep=1.5pt,circle] (7) {$0a$};
		   \draw (-4cm,.5cm) node[inner sep=1.5pt,circle] (7) {$0b$};
		   \draw (-2cm,.5cm) node[inner sep=1.5pt,circle] (7) {$1a$};
		   \draw (0cm,.5cm) node[inner sep=1.5pt,circle] (7) {$1b$};
		   \draw (-6.5cm,-7cm) node[inner sep=1.5pt,circle] (7) {$2a$};
		   \draw (-3.5cm,-7cm) node[inner sep=1.5pt,circle] (7) {$2b$};
		   \draw (-.5cm,-7cm) node[inner sep=1.5pt,circle] (7) {$2c$};
		   \draw (1.5,-7cm) node[inner sep=1.5pt,circle] (7) {$2d$};
		\end{tikzpicture}
		
	\end{center}	
	\caption{The specification of the class $\mathcal L$.}	
	\label{rootedlevel2}
\end{figure}
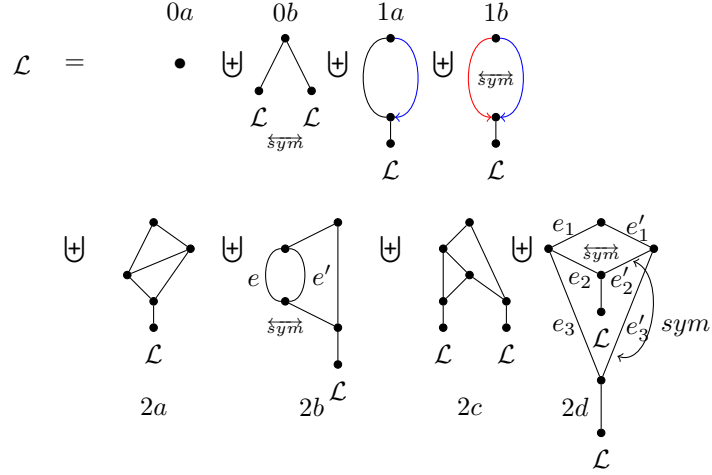

\begin{itemize}[leftmargin=2cm]
	\item[--]  We first deal with the case where the level-2 generator 
	to which the root belongs is of type $2a$. 
	This generator has $5$ internal arcs, all distinguished from each other by the structure of the generator.
	A network of $\mathcal{L}$ is attached to its reticulation vertex of out-degree 0.
	Moreover, recalling that each bridgeless component must be have at least two outgoing cut arcs, 
	at least one of the five internal arcs of the generator must carry a non-empty sequence of networks of $\mathcal{L}$. 
	Therefore, the contribution of case $2a$ to the generating function of $\mathcal{L}$ is 
	\[  {L}\cdot \sum_{i=1}^5 \binom{5}{i}  \left(\dfrac{{L}}{1-{L}}\right)^i. \]
	
	\item[--]  In the case where the level-2 generator 
	to which the root belongs is of type $2b$, we similarly have $5$ internal arcs in the generator, 
	at least one of which must be replaced by a non-empty sequence of networks of $\mathcal{L}$. 
	However, the two arcs $e$ and $e^\prime$ are not distinguishable. 
	The contribution to the generating function is therefore more subtle to analyze, 
	and we perform this detailed analysis in the Section~\ref{appendix-rooted-l2} of the Appendix. 
	The overall contribution of case $2b$ to the generating function of $\mathcal{L}$ is then shown to be 
	{\fontsize{8}{10}\selectfont
	\begin{align*} 
{L} \frac{{L}}{1-{L}} 
	+\frac{7}{2}  {L} \left(\dfrac{{L}}{1-{L}}\right)^2
	+\frac{9}{2} {L}  \left(\dfrac{{L}}{1-{L}}\right)^3 
	+ \frac52 {L}      \left(\dfrac{{L}}{1-{L}}\right)^4
    +\frac{1}{2} {L}   \left(\dfrac{{L}}{1-{L}}\right)^5.
	\end{align*}
}
\item[--] We now consider the case where the level-2 generator 
	to which the root belongs is of type $2c$. 
	This generator has $6$ internal arcs, all distinguished from each other by the structure of the generator.
	Moreover, two networks of $\mathcal{L}$ are attached to its reticulation vertices, 
	so that the condition that each bridgeless component must be have at least two outgoing cut arcs is already satisfied. 
	Therefore, all $6$ internal arcs of the generator carry possibly empty sequences of networks of $\mathcal{L}$. 
	As a consequence, the contribution of case $2c$ to the generating function of $\mathcal{L}$ is 
	$${L}^2 \left(\dfrac{1}{1-{L}}\right)^6. $$
	\item[--] Similarly, when the root belongs to a level-2 generator of type $2d$, 
the $6$ arcs of the generator carry possibly empty sequences of networks of $\mathcal{L}$. 
However, this generator enjoys both a horizontal symmetry (mapping $e_i$ to $e_i^\prime$ for $i=1,2,3$) 
and a vertical symmetry (exchanging the indices $2$ and $3$ and the corresponding pending networks of $\mathcal{L}$). 
In case all arcs carry empty sequences, the horizontal symmetry is actually the identity, so that only the vertical symmetry applies, yielding a factor $\tfrac{1}{2}$. 
Otherwise, both the horizontal and the vertical symmetry need to be taken into account, yielding a factor $\tfrac{1}{4}$. 
The total contribution of case $2d$ to the generating function of $\mathcal{L}$ is therefore 
\[ \frac{1}{2}L^2 +  \dfrac{1}{4} {L}^2\cdot \sum_{i=1}^6 \binom{6}{i} \left(\dfrac{{L}}{1-{L}}\right)^i . \]
  \end{itemize}

Following this case analysis we obtain an equation characterizing the generating function of $ \mathcal{L} $. 

\begin{theorem}\label{thm:GF_rooted_level2}
The exponential generating function $L(z)$ of rooted level-2 networks counted by number of leaves satisfies
\begin{align*}
{L}=z&+{L}^2+\frac{7{L}^2}{1-{L}}+\frac{3{L}^3}{2(1-{L})} + \frac{14{L}^3}{(1-{L})^2}+\frac{15{L}^4}{4(1-{L})^2} + \frac{29{L}^4}{2(1-{L})^3}+\frac{5{L}^5}{(1-{L})^3}\\ 
&+\frac{15{L}^5}{2(1-{L})^4}+\frac{15{L}^6}{4(1-{L})^4}+\frac{3{L}^6}{2(1-{L})^5} +\frac{3{L}^7}{2(1-{L})^5} +\frac{{L}^2}{(1-{L})^6}+\frac{{L}^8}{4(1-{L})^6},
\end{align*}
or equivalently 
\[  {L}(z)= z\phi( {L}(z)) \;\;\; where\;\;\; \phi(z)=\dfrac{1}{1-\frac{36z-102z^2+159z^3-148z^4+81z^5-24z^6+3z^7}{4(1-z)^6}}.\] 
\end{theorem}

We therefore obtain the first terms of the series expansion of $L(z)$, 
\[
 {L}(z)=z+9z^2+\frac{381}{2}z^3+\frac{20013}{4}z^4+\frac{588119}{4}z^5+\frac{37023465}{8}z^6+\cdots,
\]
as reported in Table~\ref{TableNumbers}. 

\subsection{Exact enumeration formula}

As in the previous sections, Theorem~\ref{thm:GF_rooted_level2} 
allows to derive an explicit formula for the number $\ell_n$ of rooted level-2 phylogenetic networks with $n$ leaves. This complicated formula is given in Appendix.

\subsection{Asymptotic evaluation}

Similarly, from Theorem~\ref{thm:GF_rooted_level2}, we can also derive the asymptotic behavior of $\ell_n$. 

\begin{proposition}\label{asy_rooted2}
The number $ \ell_n $ of rooted level-2 phylogenetic networks  with $ n $ leaves behaves asymptotically as
\[
\ell_n \sim c_{1}c_{2}^n n^{n-1}, 
\]
where $c_{1}\approx 0.02931$ and $ c_{2}\approx 15.433$.
\end{proposition}

\begin{proof}
Recall that 
 \[
  {L}(z)= z\phi( {L}(z)) \;\;\; where\;\;\; \phi(z)=\dfrac{1}{1-\frac{36z-102z^2+159z^3-148z^4+81z^5-24z^6+3z^7}{4(1-z)^6}}.
 \]
Denoting by $\tau \approx 0.0445$ the unique solution of the characteristic equation $\phi(z)-z\phi'(z)=0$ in the disk of convergence of $\phi$, and $\rho = \frac{\tau}{\phi(\tau)} \approx 0.0238$, 
the Singular Inversion Theorem gives:
$$[z^n] {L}(z) \sim \sqrt{\frac{\phi(\tau)}{2\phi''(\tau)}} \frac{{\rho}^{-n}}{\sqrt{\pi n^3}}\textrm{.}
$$
Like before, we get the claimed result from $\ell_n = n! [z^n] {L}(z)$ and the Stirling estimate of the factorial.
\end{proof}

\subsection{Refined enumeration formula and asymptotic distribution of parameters}
Let ${L}(z,x,y)= \sum_{n,k,m} \ell_{n,k,m} \frac{z^n}{n!} x^k y^m$ be the multivariate generating function counting rooted level-2 networks w.r.t. their number of leaves (variable $z$), their number of bridgeless components of level 1 or 2 (variable $x$) and their number of arcs across all these (variable $y$). From the specification of $\mathcal{L}$ discussed earlier, $ {L}(z,x,y)= {L} $ is easily seen to satisfy the following equation: 
\begin{align}
{L}=z&+\frac{{L}^2}{2}+x \Big( \frac{y^6{L}^2}{2} + (y^3+6y^6)\frac{{L}^2}{1-y{L}}
+\frac{3y^7{L}^3}{2(1-y{L})} + (\tfrac{y^4}{2} + \tfrac{27y^7}{2}) \frac{{L}^3}{(1-y{L})^2} \nonumber \\
& +\frac{15y^8{L}^4}{4(1-y{L})^2}  + \frac{29y^8{L}^4}{2(1-y{L})^3}
+\frac{5y^9{L}^5}{(1-y{L})^3}+\frac{15y^9{L}^5}{2(1-y{L})^4} +\frac{15y^{10}{L}^6}{4(1-y{L})^4} \nonumber \\ 
&+\frac{3y^{10}{L}^6}{2(1-y{L})^5}
 +\frac{3y^{11}{L}^7}{2(1-y{L})^5} +\frac{y^6 {L}^2}{(1-y{L})^6}  +\frac{y^{12}{L}^8}{4(1-y{L})^6}
\Big).
\label{eq:multi_rooted2}
\end{align}

From the above equation, an explicit formula for $ \ell_{n,k,m}$ could routinely be derived, as in Proposition~\ref{prop:refinedEnumerationRooted1}, although the computations are more involved. 
We decided not to report this formula here. 
Eq.~\eqref{eq:multi_rooted2} also allows to study the asymptotic behavior of the considered parameters. 

\begin{proposition}\label{asymultirooted2}
Let $X_n$ (resp. $Y_n$) be the random variable counting the number of bridgeless components of level 1 or 2 (resp. the number of edges across them) 
in rooted level-$2$ networks with $n$ leaves. 
Both $X_n$ and $Y_n$ are asymptotically normally distributed, 
and more precisely, we have 
\[
\mathbb{E} {X_n}= \mu_{X } n+ O(1), \;\;\;\;\;\; \mathbb{V}ar{X_{n}}=\sigma^{2}_{X} n+O(1) \;\;\;\;
\]
\[
\mathbb{E} {Y_{n}}= \mu_{Y} n+ O(1), \;\;\;\;\;\; \mathbb{V}ar{Y_{n}}=\sigma^{2}_{Y} n+O(1) \;\;\;\; 
\]
where $\mu_{X} \approx 0.8243$, $\sigma^{2}_{X} \approx 0.1232$, $\mu_Y \approx 4.8133 $ and $\sigma^{2}_{Y} \approx 3.5523 $. 
\label{prop:normaldistribrooted1}
\end{proposition}

\begin{proof}
To prove the result for $X_n$ (resp. $Y_n$), we specialize Eq.~\eqref{eq:multi_rooted2} for $y=1$ (resp. $x=1$) and rewrite it as 
$L(z,x,1) = F(L(z,x,1),z,x)$ for some explicit function $F$ (resp. $L(z,1,y) = F(L(z,1,y),z,y)$, for an explicit different $F$).
It is readily checked that $F$ satisfies all hypotheses of Theorem~\ref{Drmota}, as well as the solutions 
 $(L_0,z_0)$ of the system 
\begin{align*}
L=&F(L,z,1) \\
1=&F_L(L,z,1)
\end{align*} 
whose approximate values are $L_0 \approx 0.04447 $ and $z_0 \approx 0.02384$.
The result then follows from Theorem~\ref{Drmota}, 
and the numerical estimates of $\mu_X$ and $\sigma_X^2$ (resp. $\mu_Y$ and $\sigma_Y^2$)
are obtained plugging the numerical estimates for $L_0$ and $z_0$ into the explicit formulas given  by Theorem~\ref{Drmota}. 
\end{proof}

\section*{Acknowledgments}
This work was supported by a ``junior guest'' grant by the LABRI and bilateral Austrian-Taiwanese project FWF-MOST, grants
I~2309-N35 (FWF) and MOST-104-2923- M-009-006-MY3 (MOST). We thank Carine Pivoteau for her insights about random generation of combinatorial structures as well as two anonymous reviewers for their useful comments.

\bibliographystyle{alpha}      
\bibliography{bibliography}   

\newpage
\section{Appendix}

\subsection{Case analysis for unrooted level-2 generators}\label{appendix-unrooted-l2}

In the pictures below, we use thick lines to represent paths
containing at least 2 internal nodes incident with a cut-edge
which is incident with another pointed unrooted level-2 network.
We use \# to represent the fictitious root in the pointed network, 
$v$ to denote its neighbour, 
and $\mathcal{U}$ to represent any pointed network.

\subsubsection{Case 1: One edge with an attached network}

One edge of the generator carries a sequence of at least two incident cut-edges. 
Because multiple edges are not allowed, it cannot be one of the two edges incident to $v$.
So, it can be only one of the two edges not incident to $v$ (which are not distinguished). The sequence is unoriented, because of symmetry, explaining the factor $\tfrac12$ below. 

\[\frac{{U}^2}{2(1- {U})} \] 

\medskip

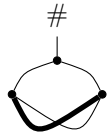
	\begin{center}
\begin{tabular}{cc}
\begin{tikzpicture}[baseline=(top.base)]
\begin{scope}[scale=0.3]
\draw (2,4.5) node [] (pointed) {\#};
\draw (2,2.5) node [circle, inner sep=1,fill,draw] (top) {};
\draw (0,1) node [circle, inner sep=1,fill,draw] (left) {};
\draw (4,1) node [circle, inner sep=1,fill,draw] (right) {};
\draw (pointed) -- (top);
\draw (top) .. controls(0.5,2) .. (left);
\draw (top) .. controls(3.5,2) .. (right);
\draw (right) .. controls(3,-1) .. (left);
\draw [line width=2pt](left) .. controls(1,-1) .. (right);
\end{scope}
\end{tikzpicture}
\end{tabular}
	\end{center}

\medskip

\subsubsection{Case 2: Two edges with attached networks}

\hspace{1em}~\\

Case 2A - Two edges of the generator carry exactly one incident cut-edge. 
Since multiple edges are not allowed, it can either be one edge incident to $v$ and one not, or both edges not incident to $v$. 
In the latter  case, the two edges should not be distinguished, hence the factor $\tfrac{1}{2}$.

 \[ {U}^2+\dfrac{{U}^2}{2}
= \frac{3}{2}{U}^2 \]

\medskip
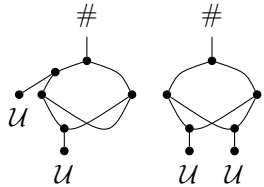
\begin{center}
\begin{tabular}{ccc}
\begin{tikzpicture}[baseline=(top.base)]
\begin{scope}[scale=0.3]
\draw (2,4.5) node [] (pointed) {\#};
\draw (2,2.5) node [circle, inner sep=1,fill,draw] (top) {};
\draw (0,1) node [circle, inner sep=1,fill,draw] (left) {};
\draw (4,1) node [circle, inner sep=1,fill,draw] (right) {};
\draw (1,-0.5) node [circle, inner sep=1,fill,draw] (leftb) {};
\draw (1,-1.5) node [circle, inner sep=1,fill,draw] (leafb) {};
\draw (1,-2.5) node {$\mathcal{U}$};
\draw (0.6,2) node [circle, inner sep=1,fill,draw] (topleft) {};
\draw (-1,1) node [circle, inner sep=1,fill,draw] (leaftopleft) {};
\draw (-1,0) node {$\mathcal{U}$};
\draw (pointed) -- (top);
\draw (top) .. controls(0.5,2) .. (left);
\draw (top) .. controls(3.5,2) .. (right);
\draw (right) .. controls(3,-1) .. (left);
\draw (left) .. controls(1,-1) .. (right);
\draw (topleft) -- (leaftopleft);
\draw (leftb) -- (leafb);
\end{scope}
\end{tikzpicture}
&
\begin{tikzpicture}[baseline=(top.base)]
\begin{scope}[scale=0.3]
\draw (2,4.5) node [] (pointed) {\#};
\draw (2,2.5) node [circle, inner sep=1,fill,draw] (top) {};
\draw (0,1) node [circle, inner sep=1,fill,draw] (left) {};
\draw (4,1) node [circle, inner sep=1,fill,draw] (right) {};
\draw (1,-0.5) node [circle, inner sep=1,fill,draw] (leftb) {};
\draw (1,-1.5) node [circle, inner sep=1,fill,draw] (leafleft) {};
\draw (1,-2.5) node {$\mathcal{U}$};
\draw (3,-0.5) node [circle, inner sep=1,fill,draw] (rightb) {};
\draw (3,-1.5) node [circle, inner sep=1,fill,draw] (leafright) {};
\draw (3,-2.5) node {$\mathcal{U}$};
\draw (pointed) -- (top);
\draw (top) .. controls(0.5,2) .. (left);
\draw (top) .. controls(3.5,2) .. (right);
\draw (right) .. controls(3,-1) .. (left);
\draw (left) .. controls(1,-1) .. (right);
\draw (rightb) -- (leafright);
\draw (leftb) -- (leafleft);
\end{scope}
\end{tikzpicture}
\end{tabular}
	\end{center}

\medskip

Case 2B - One edge of the generator carries a single incident cut-edge and another edge carries a sequence of at least two incident cut-edges. 
Again, these cannot be the two edges incident to $v$. 
The only case where symmetries need to be taken care of is when the two edges are those not incident to $v$: in this case, the sequence is not oriented, hence the factor $\tfrac{1}{2}$. 
In all other cases, the orientation of the sequence is determined by the presence of the fictitious root or the outgoing arc from the other edge with and attached network. 
 \[\dfrac{{U}^3}{1-{U}}+\dfrac{{U}^3}{1-{U}}+\dfrac{{U}^3}{2(1-{U})} 
 = \frac{5{U}^3}{2(1- {U})}\]

\medskip

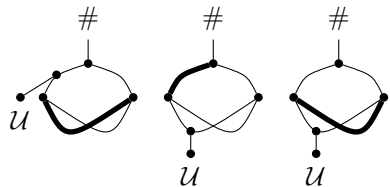
\begin{center}
\begin{tabular}{cccc}
\begin{tikzpicture}[baseline=(top.base)]
\begin{scope}[scale=0.3]
\draw (2,4.5) node [] (pointed) {\#};
\draw (2,2.5) node [circle, inner sep=1,fill,draw] (top) {};
\draw (0,1) node [circle, inner sep=1,fill,draw] (left) {};
\draw (4,1) node [circle, inner sep=1,fill,draw] (right) {};
\draw (0.6,2) node [circle, inner sep=1,fill,draw] (topleft) {};
\draw (-1,1) node [circle, inner sep=1,fill,draw] (leaftopleft) {};
\draw (-1,0) node {$\mathcal{U}$};
\draw (pointed) -- (top);
\draw (top) .. controls(0.5,2) .. (left);
\draw (top) .. controls(3.5,2) .. (right);
\draw (right) .. controls(3,-1) .. (left);
\draw [line width=2pt] (left) .. controls(1,-1) .. (right);
\draw (topleft) -- (leaftopleft);
\end{scope}
\end{tikzpicture}
&
\begin{tikzpicture}[baseline=(top.base)]
\begin{scope}[scale=0.3]
\draw (2,4.5) node [] (pointed) {\#};
\draw (2,2.5) node [circle, inner sep=1,fill,draw] (top) {};
\draw (0,1) node [circle, inner sep=1,fill,draw] (left) {};
\draw (4,1) node [circle, inner sep=1,fill,draw] (right) {};
\draw (1,-0.5) node [circle, inner sep=1,fill,draw] (leftb) {};
\draw (1,-1.5) node [circle, inner sep=1,fill,draw] (leafleft) {};
\draw (1,-2.5) node {$\mathcal{U}$};
\draw (pointed) -- (top);
\draw [line width=2pt] (top) .. controls(0.5,2) .. (left);
\draw (top) .. controls(3.5,2) .. (right);
\draw (right) .. controls(3,-1) .. (left);
\draw (left) .. controls(1,-1) .. (right);
\draw (leftb) -- (leafleft);
\end{scope}
\end{tikzpicture}
&
\begin{tikzpicture}[baseline=(top.base)]
\begin{scope}[scale=0.3]
\draw (2,4.5) node [] (pointed) {\#};
\draw (2,2.5) node [circle, inner sep=1,fill,draw] (top) {};
\draw (0,1) node [circle, inner sep=1,fill,draw] (left) {};
\draw (4,1) node [circle, inner sep=1,fill,draw] (right) {};
\draw (1,-0.5) node [circle, inner sep=1,fill,draw] (leftb) {};
\draw (1,-1.5) node [circle, inner sep=1,fill,draw] (leafleft) {};
\draw (1,-2.5) node {$\mathcal{U}$};
\draw (pointed) -- (top);
\draw (top) .. controls(0.5,2) .. (left);
\draw (top) .. controls(3.5,2) .. (right);
\draw [line width=2pt] (right) .. controls(3,-1) .. (left);
\draw (left) .. controls(1,-1) .. (right);
\draw (leftb) -- (leafleft);
\end{scope}
\end{tikzpicture}
\end{tabular}
	\end{center}

\medskip

Case 2C - Two edges of the generator (but not the two incident to $v$, as before) carry a sequence of at least two incident cut-edges. 
If one arc is incident to $v$ and the other not, then both sequences are oriented and there is no symmetry factor.
If the two arcs are those not incident to $v$, then the two sequences they carry can be seen as 
an unordered pair of oriented sequences, seen up to symmetry w.r.t. the vertical axis. This yields a factor $\tfrac{1}{2}$ since the pair is unordered, and another factor $\tfrac{1}{2}$ to account for the symmetry w.r.t. the vertical axis.

\[\dfrac{{U}^4}{(1-{U})^2}+\dfrac{{U}^4}{4(1-{U})^2}
 = \frac{5{U}^4}{4(1- {U})^2}\]

\medskip

	\begin{center}
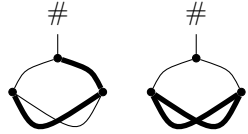
\begin{tabular}{cc}
\begin{tikzpicture}[baseline=(top.base)]
\begin{scope}[scale=0.3]
\draw (2,4.5) node [] (pointed) {\#};
\draw (2,2.5) node [circle, inner sep=1,fill,draw] (top) {};
\draw (0,1) node [circle, inner sep=1,fill,draw] (left) {};
\draw (4,1) node [circle, inner sep=1,fill,draw] (right) {};
\draw (pointed) -- (top);
\draw (top) .. controls(0.5,2) .. (left);
\draw[line width=2pt] (top) .. controls(3.5,2) .. (right);
\draw (right) .. controls(3,-1) .. (left);
\draw [line width=2pt](left) .. controls(1,-1) .. (right);
\end{scope}
\end{tikzpicture}
\;\;\;
\begin{tikzpicture}[baseline=(top.base)]
\begin{scope}[scale=0.3]
\draw (2,4.5) node [] (pointed) {\#};
\draw (2,2.5) node [circle, inner sep=1,fill,draw] (top) {};
\draw (0,1) node [circle, inner sep=1,fill,draw] (left) {};
\draw (4,1) node [circle, inner sep=1,fill,draw] (right) {};
\draw (pointed) -- (top);
\draw (top) .. controls(0.5,2) .. (left);
\draw (top) .. controls(3.5,2) .. (right);
\draw[line width=2pt] (right) .. controls(3,-1) .. (left);
\draw [line width=2pt](left) .. controls(1,-1) .. (right);
\end{scope}
\end{tikzpicture}
\end{tabular}
	\end{center}

\medskip

\subsubsection{Case 3: Three edges with attached networks}

\hspace{1em}~\\

Case 3A - Three edges of the generator carry exactly one incident cut-edge. 
The unused edge can either be incident with $v$ or not. 
In both cases, we have a factor $\tfrac{1}{2}$ because of symmetry. 

\[\dfrac{{U}^3}{2}+\dfrac{{U}^3}{2} 
= {U}^3 \]

\medskip

	\begin{center}
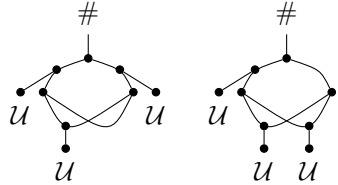
\begin{tabular}{cc}
\begin{tikzpicture}[baseline=(top.base)]
\begin{scope}[scale=0.3]
\draw (2,4.5) node [] (pointed) {\#};
\draw (2,2.5) node [circle, inner sep=1,fill,draw] (top) {};
\draw (0,1) node [circle, inner sep=1,fill,draw] (left) {};
\draw (4,1) node [circle, inner sep=1,fill,draw] (right) {};
\draw (1,-0.5) node [circle, inner sep=1,fill,draw] (leftb) {};
\draw (1,-1.5) node [circle, inner sep=1,fill,draw] (leafb) {};
\draw (1,-2.5) node {$\mathcal{U}$};
\draw (0.6,2) node [circle, inner sep=1,fill,draw] (topleft) {};
\draw (-1,1) node [circle, inner sep=1,fill,draw] (leaftopleft) {};
\draw (-1,0) node {$\mathcal{U}$};
\draw (3.4,2) node [circle, inner sep=1,fill,draw] (topright) {};
\draw (5,1) node [circle, inner sep=1,fill,draw] (leaftopright) {};
\draw (5,0) node {$\mathcal{U}$};
\draw (pointed) -- (top);
\draw (top) .. controls(0.5,2) .. (left);
\draw (top) .. controls(3.5,2) .. (right);
\draw (right) .. controls(3,-1) .. (left);
\draw (left) .. controls(1,-1) .. (right);
\draw (topleft) -- (leaftopleft);
\draw (topright) -- (leaftopright);
\draw (leftb) -- (leafb);
\end{scope}
\end{tikzpicture}
&
\begin{tikzpicture}[baseline=(top.base)]
\begin{scope}[scale=0.3]
\draw (2,4.5) node [] (pointed) {\#};
\draw (2,2.5) node [circle, inner sep=1,fill,draw] (top) {};
\draw (0,1) node [circle, inner sep=1,fill,draw] (left) {};
\draw (4,1) node [circle, inner sep=1,fill,draw] (right) {};
\draw (1,-0.5) node [circle, inner sep=1,fill,draw] (leftb) {};
\draw (1,-1.5) node [circle, inner sep=1,fill,draw] (leafleft) {};
\draw (1,-2.5) node {$\mathcal{U}$};
\draw (3,-0.5) node [circle, inner sep=1,fill,draw] (rightb) {};
\draw (3,-1.5) node [circle, inner sep=1,fill,draw] (leafright) {};
\draw (3,-2.5) node {$\mathcal{U}$};
\draw (0.6,2) node [circle, inner sep=1,fill,draw] (topleft) {};
\draw (-1,1) node [circle, inner sep=1,fill,draw] (leaftopleft) {};
\draw (-1,0) node {$\mathcal{U}$};
\draw (pointed) -- (top);
\draw (top) .. controls(0.5,2) .. (left);
\draw (top) .. controls(3.5,2) .. (right);
\draw (right) .. controls(3,-1) .. (left);
\draw (left) .. controls(1,-1) .. (right);
\draw (topleft) -- (leaftopleft);
\draw (rightb) -- (leafright);
\draw (leftb) -- (leafleft);
\end{scope}
\end{tikzpicture}
\end{tabular}
	\end{center}

\medskip

Case 3B - Two edges of the generator carry a single incident cut-edge and one carries a sequence of at least two incident cut-edges. The only cases where a symmetry comes into play here are when 
the edges carrying a single incident cut-edge are 
either the two edges incident to $v$ or the two edges not incident to $v$. 
This yield the factor $\tfrac{1}{2}$ in these two cases. 
Moreover, all sequences are oriented, because of the presence of the fictitious root 
or the single incident cut-edges.

 \[\dfrac{{U}^4}{1-{U}}+\dfrac{{U}^4}{1-{U}}+\dfrac{{U}^4}{2(1-{U})}+\dfrac{{U}^4}{2(1-{U})} 
 = \frac{3{U}^4}{1-{U}} \]

\medskip

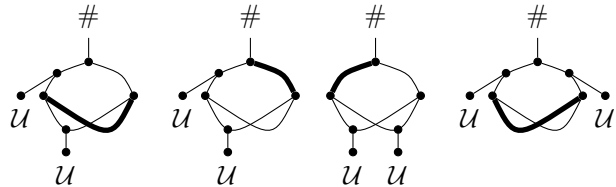
\begin{center}
\begin{tabular}{cccc}
\begin{tikzpicture}[baseline=(top.base)]
\begin{scope}[scale=0.3]
\draw (2,4.5) node [] (pointed) {\#};
\draw (2,2.5) node [circle, inner sep=1,fill,draw] (top) {};
\draw (0,1) node [circle, inner sep=1,fill,draw] (left) {};
\draw (4,1) node [circle, inner sep=1,fill,draw] (right) {};
\draw (1,-0.5) node [circle, inner sep=1,fill,draw] (leftb) {};
\draw (1,-1.5) node [circle, inner sep=1,fill,draw] (leafb) {};
\draw (1,-2.5) node {$\mathcal{U}$};
\draw (0.6,2) node [circle, inner sep=1,fill,draw] (topleft) {};
\draw (-1,1) node [circle, inner sep=1,fill,draw] (leaftopleft) {};
\draw (-1,0) node {$\mathcal{U}$};
\draw (pointed) -- (top);
\draw (top) .. controls(0.5,2) .. (left);
\draw (top) .. controls(3.5,2) .. (right);
\draw [line width=2pt] (right) .. controls(3,-1) .. (left);
\draw (left) .. controls(1,-1) .. (right);
\draw (topleft) -- (leaftopleft);
\draw (leftb) -- (leafb);
\end{scope}
\end{tikzpicture}
&
\begin{tikzpicture}[baseline=(top.base)]
\begin{scope}[scale=0.3]
\draw (2,4.5) node [] (pointed) {\#};
\draw (2,2.5) node [circle, inner sep=1,fill,draw] (top) {};
\draw (0,1) node [circle, inner sep=1,fill,draw] (left) {};
\draw (4,1) node [circle, inner sep=1,fill,draw] (right) {};
\draw (1,-0.5) node [circle, inner sep=1,fill,draw] (leftb) {};
\draw (1,-1.5) node [circle, inner sep=1,fill,draw] (leafb) {};
\draw (1,-2.5) node {$\mathcal{U}$};
\draw (0.6,2) node [circle, inner sep=1,fill,draw] (topleft) {};
\draw (-1,1) node [circle, inner sep=1,fill,draw] (leaftopleft) {};
\draw (-1,0) node {$\mathcal{U}$};
\draw (pointed) -- (top);
\draw (top) .. controls(0.5,2) .. (left);
\draw [line width=2pt] (top) .. controls(3.5,2) .. (right);
\draw (right) .. controls(3,-1) .. (left);
\draw (left) .. controls(1,-1) .. (right);
\draw (topleft) -- (leaftopleft);
\draw (leftb) -- (leafb);
\end{scope}
\end{tikzpicture}
&
\begin{tikzpicture}[baseline=(top.base)]
\begin{scope}[scale=0.3]
\draw (2,4.5) node [] (pointed) {\#};
\draw (2,2.5) node [circle, inner sep=1,fill,draw] (top) {};
\draw (0,1) node [circle, inner sep=1,fill,draw] (left) {};
\draw (4,1) node [circle, inner sep=1,fill,draw] (right) {};
\draw (1,-0.5) node [circle, inner sep=1,fill,draw] (leftb) {};
\draw (1,-1.5) node [circle, inner sep=1,fill,draw] (leafleft) {};
\draw (1,-2.5) node {$\mathcal{U}$};
\draw (3,-0.5) node [circle, inner sep=1,fill,draw] (rightb) {};
\draw (3,-1.5) node [circle, inner sep=1,fill,draw] (leafright) {};
\draw (3,-2.5) node {$\mathcal{U}$};
\draw (pointed) -- (top);
\draw  [line width=2pt] (top) .. controls(0.5,2) .. (left);
\draw (top) .. controls(3.5,2) .. (right);
\draw (right) .. controls(3,-1) .. (left);
\draw (left) .. controls(1,-1) .. (right);
\draw (rightb) -- (leafright);
\draw (leftb) -- (leafleft);
\end{scope}
\end{tikzpicture}
&
\begin{tikzpicture}[baseline=(top.base)]
\begin{scope}[scale=0.3]
\draw (2,4.5) node [] (pointed) {\#};
\draw (2,2.5) node [circle, inner sep=1,fill,draw] (top) {};
\draw (0,1) node [circle, inner sep=1,fill,draw] (left) {};
\draw (4,1) node [circle, inner sep=1,fill,draw] (right) {};
\draw (0.6,2) node [circle, inner sep=1,fill,draw] (topleft) {};
\draw (-1,1) node [circle, inner sep=1,fill,draw] (leaftopleft) {};
\draw (-1,0) node {$\mathcal{U}$};
\draw (3.4,2) node [circle, inner sep=1,fill,draw] (topright) {};
\draw (5,1) node [circle, inner sep=1,fill,draw] (leaftopright) {};
\draw (5,0) node {$\mathcal{U}$};
\draw (pointed) -- (top);
\draw (top) .. controls(0.5,2) .. (left);
\draw (top) .. controls(3.5,2) .. (right);
\draw (right) .. controls(3,-1) .. (left);
\draw  [line width=2pt] (left) .. controls(1,-1) .. (right);
\draw (topleft) -- (leaftopleft);
\draw (topright) -- (leaftopright);
\end{scope}
\end{tikzpicture}
\end{tabular}
	\end{center}

\medskip

Case 3C -  One edge of the generator carries a single incident cut-edge and two edges carry a sequence of at least two incident cut-edges. Similarly to the previous case, 
we obtain a factor $\tfrac{1}{2}$ for symmetry reasons when 
the two edges carrying sequences are 
either the two edges incident to $v$ or the two edges not incident to $v$. 
Moreover, all sequences are oriented, because of the presence of the fictitious root 
or the single incident cut-edge. 
\[\dfrac{{U}^5}{(1-{U})^2}+\dfrac{{U}^5}{(1-{U})^2}+\dfrac{{U}^5}{2(1-{U})^2}+\dfrac{{U}^5}{2(1-{U})^2} 
 = \frac{3{U}^5}{(1- {U})^2} \]


\medskip

	\begin{center}
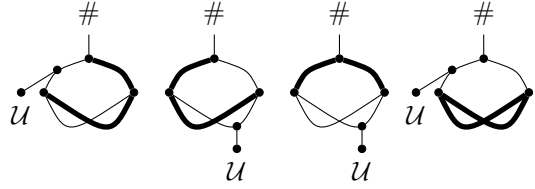
\begin{tabular}{ccc}
\begin{tikzpicture}[baseline=(top.base)]
\begin{scope}[scale=0.3]
\draw (2,4.5) node [] (pointed) {\#};
\draw (2,2.5) node [circle, inner sep=1,fill,draw] (top) {};
\draw (0,1) node [circle, inner sep=1,fill,draw] (left) {};
\draw (4,1) node [circle, inner sep=1,fill,draw] (right) {};
\draw (0.6,2) node [circle, inner sep=1,fill,draw] (topleft) {};
\draw (-1,1) node [circle, inner sep=1,fill,draw] (leaftopleft) {};
\draw (-1,0) node {$\mathcal{U}$};
\draw (pointed) -- (top);
\draw (top) .. controls(0.5,2) .. (left);
\draw[line width=2pt] (top) .. controls(3.5,2) .. (right);
\draw[line width=2pt] (right) .. controls(3,-1) .. (left);
\draw (left) .. controls(1,-1) .. (right);
\draw (topleft) -- (leaftopleft);
\end{scope}
\end{tikzpicture}
&
\begin{tikzpicture}[baseline=(top.base)]
\begin{scope}[scale=0.3]
\draw (2,4.5) node [] (pointed) {\#};
\draw (2,2.5) node [circle, inner sep=1,fill,draw] (top) {};
\draw (0,1) node [circle, inner sep=1,fill,draw] (left) {};
\draw (4,1) node [circle, inner sep=1,fill,draw] (right) {};
\draw (3,-0.5) node [circle, inner sep=1,fill,draw] (rightb) {};
\draw (3,-1.5) node [circle, inner sep=1,fill,draw] (leafright) {};
\draw (3,-2.5) node {$\mathcal{U}$};
\draw (pointed) -- (top);
\draw[line width=2pt] (top) .. controls(0.5,2) .. (left);
\draw (top) .. controls(3.5,2) .. (right);
\draw (right) .. controls(3,-1) .. (left);
\draw[line width=2pt] (left) .. controls(1,-1) .. (right);
\draw (rightb) -- (leafright);
\end{scope}
\end{tikzpicture}
&
\begin{tikzpicture}[baseline=(top.base)]
\begin{scope}[scale=0.3]
\draw (2,4.5) node [] (pointed) {\#};
\draw (2,2.5) node [circle, inner sep=1,fill,draw] (top) {};
\draw (0,1) node [circle, inner sep=1,fill,draw] (left) {};
\draw (4,1) node [circle, inner sep=1,fill,draw] (right) {};
\draw (3,-0.5) node [circle, inner sep=1,fill,draw] (rightb) {};
\draw (3,-1.5) node [circle, inner sep=1,fill,draw] (leafright) {};
\draw (3,-2.5) node {$\mathcal{U}$};
\draw (pointed) -- (top);
\draw[line width=2pt] (top) .. controls(0.5,2) .. (left);
\draw[line width=2pt] (top) .. controls(3.5,2) .. (right);
\draw (right) .. controls(3,-1) .. (left);
\draw (left) .. controls(1,-1) .. (right);
\draw (rightb) -- (leafright);
\end{scope}
\end{tikzpicture}

\begin{tikzpicture}[baseline=(top.base)]
\begin{scope}[scale=0.3]
\draw (2,4.5) node [] (pointed) {\#};
\draw (2,2.5) node [circle, inner sep=1,fill,draw] (top) {};
\draw (0,1) node [circle, inner sep=1,fill,draw] (left) {};
\draw (4,1) node [circle, inner sep=1,fill,draw] (right) {};
\draw (0.6,2) node [circle, inner sep=1,fill,draw] (topleft) {};
\draw (-1,1) node [circle, inner sep=1,fill,draw] (leaftopleft) {};
\draw (-1,0) node {$\mathcal{U}$};
\draw (pointed) -- (top);
\draw (top) .. controls(0.5,2) .. (left);
\draw (top) .. controls(3.5,2) .. (right);
\draw[line width=2pt] (right) .. controls(3,-1) .. (left);
\draw[line width=2pt] (left) .. controls(1,-1) .. (right);
\draw (topleft) -- (leaftopleft);
\end{scope}
\end{tikzpicture}
\end{tabular}
	\end{center}

\medskip

Case 3D - Three edges of the generator carry a sequence of at least two incident cut-edges.
In both cases, we have a factor $\tfrac12$ for symmetry reason, but all sequences are oriented 
by the presence of the fictitious root, or of the sequence on the edge(s) incident to $v$. 

 \[\dfrac{{U}^6}{2(1-{U})^3}+\dfrac{{U}^6}{2(1-{U})^3} 
 = \frac{{U}^6}{(1- {U})^3} \]
 

\medskip

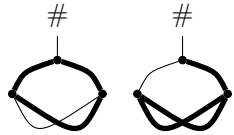
\begin{center}
\begin{tabular}{cc}
\begin{tikzpicture}[baseline=(top.base)]
\begin{scope}[scale=0.3]
\draw (2,4.5) node [] (pointed) {\#};
\draw (2,2.5) node [circle, inner sep=1,fill,draw] (top) {};
\draw (0,1) node [circle, inner sep=1,fill,draw] (left) {};
\draw (4,1) node [circle, inner sep=1,fill,draw] (right) {};
\draw (pointed) -- (top);
\draw[line width=2pt] (top) .. controls(0.5,2) .. (left);
\draw[line width=2pt] (top) .. controls(3.5,2) .. (right);
\draw[line width=2pt] (right) .. controls(3,-1) .. (left);
\draw (left) .. controls(1,-1) .. (right);
\end{scope}
\end{tikzpicture}
&
\begin{tikzpicture}[baseline=(top.base)]
\begin{scope}[scale=0.3]
\draw (2,4.5) node [] (pointed) {\#};
\draw (2,2.5) node [circle, inner sep=1,fill,draw] (top) {};
\draw (0,1) node [circle, inner sep=1,fill,draw] (left) {};
\draw (4,1) node [circle, inner sep=1,fill,draw] (right) {};
\draw (pointed) -- (top);
\draw (top) .. controls(0.5,2) .. (left);
\draw[line width=2pt] (top) .. controls(3.5,2) .. (right);
\draw[line width=2pt] (right) .. controls(3,-1) .. (left);
\draw [line width=2pt](left) .. controls(1,-1) .. (right);
\end{scope}
\end{tikzpicture}
\end{tabular}
	\end{center}

\medskip

\subsubsection{Case 4: Four edges with attached networks}

\hspace{1em}~\\

Case 4A - The four edges of the generator each carry exactly one incident cut-edge. 
In this case, the two edges incident to $v$ can be exchanged without modifying the network, 
and the same holds for the two edges not incident to $v$. 
This yields a factor $\tfrac{1}{2} \cdot \tfrac{1}{2} = \tfrac{1}{4}$ due to symmetries. 

\[\dfrac{{U}^4}{4} \]

\medskip

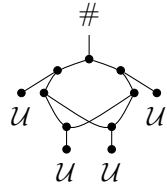
\begin{center}
\begin{tikzpicture}[baseline=(top.base)]
\begin{scope}[scale=0.3]
\draw (2,4.5) node [] (pointed) {\#};
\draw (2,2.5) node [circle, inner sep=1,fill,draw] (top) {};
\draw (0,1) node [circle, inner sep=1,fill,draw] (left) {};
\draw (4,1) node [circle, inner sep=1,fill,draw] (right) {};
\draw (1,-0.5) node [circle, inner sep=1,fill,draw] (leftb) {};
\draw (1,-1.5) node [circle, inner sep=1,fill,draw] (leafb) {};
\draw (1,-2.5) node {$\mathcal{U}$};
\draw (0.6,2) node [circle, inner sep=1,fill,draw] (topleft) {};
\draw (-1,1) node [circle, inner sep=1,fill,draw] (leaftopleft) {};
\draw (-1,0) node {$\mathcal{U}$};
\draw (3.4,2) node [circle, inner sep=1,fill,draw] (topright) {};
\draw (5,1) node [circle, inner sep=1,fill,draw] (leaftopright) {};
\draw (5,0) node {$\mathcal{U}$};
\draw (3,-0.5) node [circle, inner sep=1,fill,draw] (rightb) {};
\draw (3,-1.5) node [circle, inner sep=1,fill,draw] (leafright) {};
\draw (3,-2.5) node {$\mathcal{U}$};
\draw (pointed) -- (top);
\draw (top) .. controls(0.5,2) .. (left);
\draw (top) .. controls(3.5,2) .. (right);
\draw (right) .. controls(3,-1) .. (left);
\draw (left) .. controls(1,-1) .. (right);
\draw (topleft) -- (leaftopleft);
\draw (topright) -- (leaftopright);
\draw (leftb) -- (leafb);
\draw (rightb) -- (leafright);
\end{scope}
\end{tikzpicture}
	\end{center}

\medskip

Case 4B - Three edges of the generator carry a single incident cut-edge and the fourth one carries a sequence of at least two incident cut-edges. If this fourth edge is one incident to $v$, then the sequence it carries is oriented by the presence of the fictitious root, but the two arcs pending on the edges not incident to $v$ are symmetric, hence a factor $\tfrac{1}{2}$. 
If on the contrary the edge carrying the sequence is not incident to $v$, then the sequence is also oriented, this time because of the arcs attached to the edges incident to $v$. Moreover, the picture has a symmetry w.r.t. the vertical axis, hence a factor $\tfrac{1}{2}$. 

 \[\dfrac{{U}^5}{2(1-{U})}+\dfrac{{U}^5}{2(1-{U})} 
= \frac{{U}^5}{1- {U}} \]

\medskip

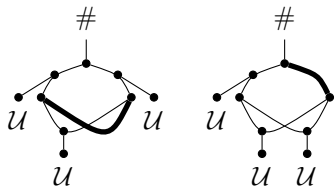
\begin{center}
\begin{tabular}{cc}
\begin{tikzpicture}[baseline=(top.base)]
\begin{scope}[scale=0.3]
\draw (2,4.5) node [] (pointed) {\#};
\draw (2,2.5) node [circle, inner sep=1,fill,draw] (top) {};
\draw (0,1) node [circle, inner sep=1,fill,draw] (left) {};
\draw (4,1) node [circle, inner sep=1,fill,draw] (right) {};
\draw (1,-0.5) node [circle, inner sep=1,fill,draw] (leftb) {};
\draw (1,-1.5) node [circle, inner sep=1,fill,draw] (leafb) {};
\draw (1,-2.5) node {$\mathcal{U}$};
\draw (0.6,2) node [circle, inner sep=1,fill,draw] (topleft) {};
\draw (-1,1) node [circle, inner sep=1,fill,draw] (leaftopleft) {};
\draw (-1,0) node {$\mathcal{U}$};
\draw (3.4,2) node [circle, inner sep=1,fill,draw] (topright) {};
\draw (5,1) node [circle, inner sep=1,fill,draw] (leaftopright) {};
\draw (5,0) node {$\mathcal{U}$};
\draw (pointed) -- (top);
\draw (top) .. controls(0.5,2) .. (left);
\draw (top) .. controls(3.5,2) .. (right);
\draw (right)[line width=2pt] .. controls(3,-1) .. (left);
\draw (left) .. controls(1,-1) .. (right);
\draw (topleft) -- (leaftopleft);
\draw (topright) -- (leaftopright);
\draw (leftb) -- (leafb);
\end{scope}
\end{tikzpicture}
&
\begin{tikzpicture}[baseline=(top.base)]
\begin{scope}[scale=0.3]
\draw (2,4.5) node [] (pointed) {\#};
\draw (2,2.5) node [circle, inner sep=1,fill,draw] (top) {};
\draw (0,1) node [circle, inner sep=1,fill,draw] (left) {};
\draw (4,1) node [circle, inner sep=1,fill,draw] (right) {};
\draw (1,-0.5) node [circle, inner sep=1,fill,draw] (leftb) {};
\draw (1,-1.5) node [circle, inner sep=1,fill,draw] (leafleft) {};
\draw (1,-2.5) node {$\mathcal{U}$};
\draw (3,-0.5) node [circle, inner sep=1,fill,draw] (rightb) {};
\draw (3,-1.5) node [circle, inner sep=1,fill,draw] (leafright) {};
\draw (3,-2.5) node {$\mathcal{U}$};
\draw (0.6,2) node [circle, inner sep=1,fill,draw] (topleft) {};
\draw (-1,1) node [circle, inner sep=1,fill,draw] (leaftopleft) {};
\draw (-1,0) node {$\mathcal{U}$};
\draw (pointed) -- (top);
\draw (top) .. controls(0.5,2) .. (left);
\draw (top)[line width=2pt] .. controls(3.5,2) .. (right);
\draw (right) .. controls(3,-1) .. (left);
\draw (left) .. controls(1,-1) .. (right);
\draw (topleft) -- (leaftopleft);
\draw (rightb) -- (leafright);
\draw (leftb) -- (leafleft);
\end{scope}
\end{tikzpicture}
\end{tabular}
	\end{center}

\medskip

Case 4C - Two edges carry a single incident cut-edge and the two others carry a sequence of at least two incident cut-edges. In all cases, the sequences are oriented, by the presence of either the fictitious root or of the single arcs attached to edges. If the edges carrying sequences are one incident to $v$ and the other not incident to $v$, all edges are in addition distinguished from each other. In the other two cases, both edges incident to $v$ form an unordered pair, as well as the two edges not incident to $v$. In each case, we therefore have a factor $\tfrac{1}{4}$.
\[\dfrac{{U}^6}{(1-{U})^2}+\dfrac{{U}^6}{4(1-{U})^2}+\dfrac{{U}^6}{4(1-{U})^2} 
 = \frac{3{U}^6}{2(1- {U})^2} \]


\medskip

	\begin{center}
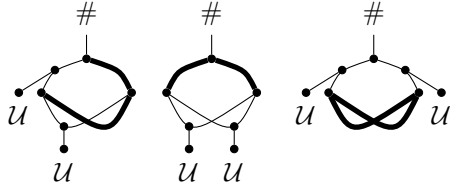
\begin{tabular}{ccc}
\begin{tikzpicture}[baseline=(top.base)]
\begin{scope}[scale=0.3]
\draw (2,4.5) node [] (pointed) {\#};
\draw (2,2.5) node [circle, inner sep=1,fill,draw] (top) {};
\draw (0,1) node [circle, inner sep=1,fill,draw] (left) {};
\draw (4,1) node [circle, inner sep=1,fill,draw] (right) {};
\draw (1,-0.5) node [circle, inner sep=1,fill,draw] (leftb) {};
\draw (1,-1.5) node [circle, inner sep=1,fill,draw] (leafb) {};
\draw (1,-2.5) node {$\mathcal{U}$};
\draw (0.6,2) node [circle, inner sep=1,fill,draw] (topleft) {};
\draw (-1,1) node [circle, inner sep=1,fill,draw] (leaftopleft) {};
\draw (-1,0) node {$\mathcal{U}$};
\draw (pointed) -- (top);
\draw (top) .. controls(0.5,2) .. (left);
\draw (top)[line width=2pt] .. controls(3.5,2) .. (right);
\draw (right)[line width=2pt] .. controls(3,-1) .. (left);
\draw (left) .. controls(1,-1) .. (right);
\draw (topleft) -- (leaftopleft);
\draw (leftb) -- (leafb);
\end{scope}
\end{tikzpicture}
&
\begin{tikzpicture}[baseline=(top.base)]
\begin{scope}[scale=0.3]
\draw (2,4.5) node [] (pointed) {\#};
\draw (2,2.5) node [circle, inner sep=1,fill,draw] (top) {};
\draw (0,1) node [circle, inner sep=1,fill,draw] (left) {};
\draw (4,1) node [circle, inner sep=1,fill,draw] (right) {};
\draw (1,-0.5) node [circle, inner sep=1,fill,draw] (leftb) {};
\draw (1,-1.5) node [circle, inner sep=1,fill,draw] (leafleft) {};
\draw (1,-2.5) node {$\mathcal{U}$};
\draw (3,-0.5) node [circle, inner sep=1,fill,draw] (rightb) {};
\draw (3,-1.5) node [circle, inner sep=1,fill,draw] (leafright) {};
\draw (3,-2.5) node {$\mathcal{U}$};
\draw (pointed) -- (top);
\draw (top)[line width=2pt] .. controls(0.5,2) .. (left);
\draw (top)[line width=2pt] .. controls(3.5,2) .. (right);
\draw (right) .. controls(3,-1) .. (left);
\draw (left) .. controls(1,-1) .. (right);
\draw (rightb) -- (leafright);
\draw (leftb) -- (leafleft);
\end{scope}
\end{tikzpicture}
&
\begin{tikzpicture}[baseline=(top.base)]
\begin{scope}[scale=0.3]
\draw (2,4.5) node [] (pointed) {\#};
\draw (2,2.5) node [circle, inner sep=1,fill,draw] (top) {};
\draw (0,1) node [circle, inner sep=1,fill,draw] (left) {};
\draw (4,1) node [circle, inner sep=1,fill,draw] (right) {};
\draw (0.6,2) node [circle, inner sep=1,fill,draw] (topleft) {};
\draw (-1,1) node [circle, inner sep=1,fill,draw] (leaftopleft) {};
\draw (-1,0) node {$\mathcal{U}$};
\draw (3.4,2) node [circle, inner sep=1,fill,draw] (topright) {};
\draw (5,1) node [circle, inner sep=1,fill,draw] (leaftopright) {};
\draw (5,0) node {$\mathcal{U}$};
\draw (pointed) -- (top);
\draw (top) .. controls(0.5,2) .. (left);
\draw (top) .. controls(3.5,2) .. (right);
\draw (right)[line width=2pt] .. controls(3,-1) .. (left);
\draw (left)[line width=2pt] .. controls(1,-1) .. (right);
\draw (topleft) -- (leaftopleft);
\draw (topright) -- (leaftopright);
\end{scope}
\end{tikzpicture}
\end{tabular}
	\end{center}

\medskip

Case 4D - One edge of the generator carries a single incident cut-edge and three edges carry a sequence of at least two incident cut-edges. As in the previous case, all sequences are oriented. However, if the two edges incident to $v$ carry a sequence, the picture has a symmetry w.r.t. the vertical axis, hence a factor $\tfrac{1}{2}$. If on the contrary the two edges not incident to $v$ carry a sequence, these two edges are indistinguishable, hence a factor $\tfrac{1}{2}$ also in this case. 

 \[\dfrac{{U}^7}{2(1-{U})^3}+\dfrac{{U}^7}{2(1-{U})^3} 
 = \frac{{U}^7}{(1- {U})^3} \]


\medskip

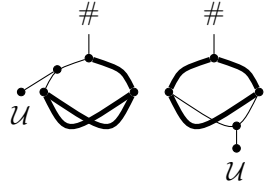
\begin{center}
\begin{tabular}{ccc}
\begin{tikzpicture}[baseline=(top.base)]
\begin{scope}[scale=0.3]
\draw (2,4.5) node [] (pointed) {\#};
\draw (2,2.5) node [circle, inner sep=1,fill,draw] (top) {};
\draw (0,1) node [circle, inner sep=1,fill,draw] (left) {};
\draw (4,1) node [circle, inner sep=1,fill,draw] (right) {};
\draw (0.6,2) node [circle, inner sep=1,fill,draw] (topleft) {};
\draw (-1,1) node [circle, inner sep=1,fill,draw] (leaftopleft) {};
\draw (-1,0) node {$\mathcal{U}$};
\draw (pointed) -- (top);
\draw (top) .. controls(0.5,2) .. (left);
\draw[line width=2pt] (top) .. controls(3.5,2) .. (right);
\draw[line width=2pt] (right) .. controls(3,-1) .. (left);
\draw[line width=2pt] (left) .. controls(1,-1) .. (right);
\draw (topleft) -- (leaftopleft);
\end{scope}
\end{tikzpicture}
&
\begin{tikzpicture}[baseline=(top.base)]
\begin{scope}[scale=0.3]
\draw (2,4.5) node [] (pointed) {\#};
\draw (2,2.5) node [circle, inner sep=1,fill,draw] (top) {};
\draw (0,1) node [circle, inner sep=1,fill,draw] (left) {};
\draw (4,1) node [circle, inner sep=1,fill,draw] (right) {};
\draw (3,-0.5) node [circle, inner sep=1,fill,draw] (rightb) {};
\draw (3,-1.5) node [circle, inner sep=1,fill,draw] (leafright) {};
\draw (3,-2.5) node {$\mathcal{U}$};
\draw (pointed) -- (top);
\draw[line width=2pt] (top) .. controls(0.5,2) .. (left);
\draw[line width=2pt] (top) .. controls(3.5,2) .. (right);
\draw (right) .. controls(3,-1) .. (left);
\draw[line width=2pt] (left) .. controls(1,-1) .. (right);
\draw (rightb) -- (leafright);
\end{scope}
\end{tikzpicture}
\end{tabular}
	\end{center}
	
\medskip

Case 4E - All four edges of the generator carry a sequence of at least two incident cut-edges. Then all sequences are oriented, but the two edges not incident to $v$ are indistinguishable. The picture has in addition a symmetry w.r.t. the vertical axis. This yields a factor $\tfrac{1}{4}$. 

\[\dfrac{{U}^8}{4(1-{U})^4} \]
 

\medskip

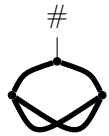
\begin{center}
\begin{tabular}{cc}
\begin{tikzpicture}[baseline=(top.base)]
\begin{scope}[scale=0.3]
\draw (2,4.5) node [] (pointed) {\#};
\draw (2,2.5) node [circle, inner sep=1,fill,draw] (top) {};
\draw (0,1) node [circle, inner sep=1,fill,draw] (left) {};
\draw (4,1) node [circle, inner sep=1,fill,draw] (right) {};
\draw (pointed) -- (top);
\draw[line width=2pt] (top) .. controls(0.5,2) .. (left);
\draw[line width=2pt] (top) .. controls(3.5,2) .. (right);
\draw[line width=2pt] (right) .. controls(3,-1) .. (left);
\draw[line width=2pt] (left) .. controls(1,-1) .. (right);
\end{scope}
\end{tikzpicture}
\end{tabular}
	\end{center}


\subsection{Case analysis for the rooted level-2 generator 2b}\label{appendix-rooted-l2}

In the pictures below, we use thick lines to represent paths
containing at least one internal node incident with a cut arc
which is incident with the root of another rooted level-2 network.
All arcs are directed downwards. 
We use $\mathcal{L}$ to represented any rooted level-2 network. 

\subsubsection{Case 1: } Only one arc of the generator
carries a sequence of at least one outgoing arc.
This arc can only be $e$ or $e'$ (and these cases are indistinguishable), since otherwise the network would contain multiple arcs, and this is not allowed.

$$   {L} \frac{{L}}{1-{L}}  
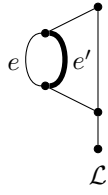
$$

\begin{figure}[!h]
	\begin{center}
	  \begin{tabular}{cccc}
		\begin{tikzpicture}[
		scale=0.7,
		level/.style={thick},
		virtual/.style={thick,densely dashed},
		trans/.style={thick,<->,shorten >=2pt,shorten <=2pt,>=stealth},
		classical/.style={thin,double,<->,shorten >=4pt,shorten <=4pt,>=stealth}
		]
	\draw (-3cm,-3.5cm) node[circle,inner sep=1,fill,draw] (1) {};
	\draw (-4cm,-4cm) node[circle,inner sep=1,fill,draw] (2) {};
	\draw (-4cm,-5cm) node[circle,inner sep=1,fill,draw] (3) {};
	\draw (-3cm,-5.5cm) node[circle,inner sep=1,fill,draw] (4) {};
	\draw (-3cm,-6.2cm) node[circle,inner sep=1,fill,draw] (5) {};
	\draw (-3cm,-6.7cm) node[inner sep=1.5pt,circle] (6) {${\mathcal L}$};
	\draw (1)--(2);\draw (1)--(4);\draw (3)--(4);\draw (4)--(5);
	\draw  [line width=2pt] (2) to[out=0,in=0] (3);
	\draw  [black,-] (2) to[out=180,in=180] (3);
	\draw (-4.6cm,-4.6cm) node[inner sep=1.5pt,circle] (7) {$e$};
	\draw (-3.3cm,-4.5cm) node[inner sep=1.5pt,circle] (8) {$e^\prime$};
		\end{tikzpicture} 
      \end{tabular}
	\end{center}	
\end{figure}

\subsubsection{Case 2: } Exactly two arcs of the generator
carry a sequence of at least one outgoing arc.
To avoid multiple arcs, either these two arcs are $e$ and $e'$ (and those two arcs
are symmetric, hence the factor $\frac12$),
or one of them is $e$ or $e'$ (which are not distinguished) and the other arc is chosen among the three arcs different from $e$ and $e'$.

$$ \frac{1}{2}  {L} \left(\dfrac{{L}}{1-{L}}\right)^2 + 3 {L} \left(\dfrac{{L}}{1-{L}}\right)^2 = 
\frac{7}{2}  {L} \left(\dfrac{{L}}{1-{L}}\right)^2$$

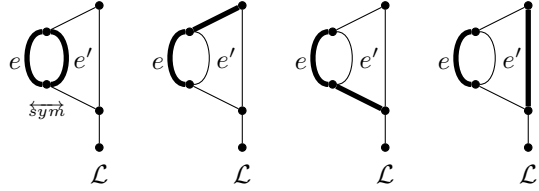
\begin{figure}[!h]
	\begin{center}
	  \begin{tabular}{ccccccc}
		\begin{tikzpicture}[
		scale=0.7,
		level/.style={thick},
		virtual/.style={thick,densely dashed},
		trans/.style={thick,<->,shorten >=2pt,shorten <=2pt,>=stealth},
		classical/.style={thin,double,<->,shorten >=4pt,shorten <=4pt,>=stealth}
		]
	\draw (-3cm,-3.5cm) node[circle,inner sep=1,fill,draw] (1) {};
	\draw (-4cm,-4cm) node[circle,inner sep=1,fill,draw] (2) {};
	\draw (-4cm,-5cm) node[circle,inner sep=1,fill,draw] (3) {};
	\draw (-3cm,-5.5cm) node[circle,inner sep=1,fill,draw] (4) {};
	\draw (-3cm,-6.2cm) node[circle,inner sep=1,fill,draw] (5) {};
	\draw (-3cm,-6.7cm) node[inner sep=1.5pt,circle] (6) {${\mathcal L}$};
	\draw (1)--(2);\draw (1)--(4);\draw (3)--(4);\draw (4)--(5);
	\draw [line width=2pt] (2) to[out=0,in=0] (3);
	\draw [line width=2pt] (2) to[out=180,in=180] (3);
	\draw (-4.6cm,-4.6cm) node[inner sep=1.5pt,circle] (7) {$e$};
	\draw (-3.3cm,-4.5cm) node[inner sep=1.5pt,circle] (8) {$e^\prime$};
	\draw (-4cm,-5.5cm) node[inner sep=1.5pt,circle] (9) {$\tiny{\overleftrightarrow{{sym}}}$};
		\end{tikzpicture} &

		\begin{tikzpicture}[
		scale=0.7,
		level/.style={thick},
		virtual/.style={thick,densely dashed},
		trans/.style={thick,<->,shorten >=2pt,shorten <=2pt,>=stealth},
		classical/.style={thin,double,<->,shorten >=4pt,shorten <=4pt,>=stealth}
		]
	\draw (-3cm,-3.5cm) node[circle,inner sep=1,fill,draw] (1) {};
	\draw (-4cm,-4cm) node[circle,inner sep=1,fill,draw] (2) {};
	\draw (-4cm,-5cm) node[circle,inner sep=1,fill,draw] (3) {};
	\draw (-3cm,-5.5cm) node[circle,inner sep=1,fill,draw] (4) {};
	\draw (-3cm,-6.2cm) node[circle,inner sep=1,fill,draw] (5) {};
	\draw (-3cm,-6.7cm) node[inner sep=1.5pt,circle] (6) {${\mathcal L}$};
	\draw [line width=2pt] (1)--(2);\draw (1)--(4);\draw (3)--(4);\draw (4)--(5);
	\draw  [black,-] (2) to[out=0,in=0] (3);
	\draw  [line width=2pt] (2) to[out=180,in=180] (3);
	\draw (-4.6cm,-4.6cm) node[inner sep=1.5pt,circle] (7) {$e$};
	\draw (-3.3cm,-4.5cm) node[inner sep=1.5pt,circle] (8) {$e^\prime$};
		\end{tikzpicture} &

		\begin{tikzpicture}[
		scale=0.7,
		level/.style={thick},
		virtual/.style={thick,densely dashed},
		trans/.style={thick,<->,shorten >=2pt,shorten <=2pt,>=stealth},
		classical/.style={thin,double,<->,shorten >=4pt,shorten <=4pt,>=stealth}
		]
	\draw (-3cm,-3.5cm) node[circle,inner sep=1,fill,draw] (1) {};
	\draw (-4cm,-4cm) node[circle,inner sep=1,fill,draw] (2) {};
	\draw (-4cm,-5cm) node[circle,inner sep=1,fill,draw] (3) {};
	\draw (-3cm,-5.5cm) node[circle,inner sep=1,fill,draw] (4) {};
	\draw (-3cm,-6.2cm) node[circle,inner sep=1,fill,draw] (5) {};
	\draw (-3cm,-6.7cm) node[inner sep=1.5pt,circle] (6) {${\mathcal L}$};
	\draw (1)--(2);\draw (1)--(4);\draw [line width=2pt] (3)--(4);\draw (4)--(5);
	\draw  [black,-] (2) to[out=0,in=0] (3);
	\draw [line width=2pt] (2) to[out=180,in=180] (3);
	\draw (-4.6cm,-4.6cm) node[inner sep=1.5pt,circle] (7) {$e$};
	\draw (-3.3cm,-4.5cm) node[inner sep=1.5pt,circle] (8) {$e^\prime$};
		\end{tikzpicture} &

		\begin{tikzpicture}[
		scale=0.7,
		level/.style={thick},
		virtual/.style={thick,densely dashed},
		trans/.style={thick,<->,shorten >=2pt,shorten <=2pt,>=stealth},
		classical/.style={thin,double,<->,shorten >=4pt,shorten <=4pt,>=stealth}
		]
	\draw (-3cm,-3.5cm) node[circle,inner sep=1,fill,draw] (1) {};
	\draw (-4cm,-4cm) node[circle,inner sep=1,fill,draw] (2) {};
	\draw (-4cm,-5cm) node[circle,inner sep=1,fill,draw] (3) {};
	\draw (-3cm,-5.5cm) node[circle,inner sep=1,fill,draw] (4) {};
	\draw (-3cm,-6.2cm) node[circle,inner sep=1,fill,draw] (5) {};
	\draw (-3cm,-6.7cm) node[inner sep=1.5pt,circle] (6) {${\mathcal L}$};
	\draw (1)--(2);\draw [line width=2pt] (1)--(4);\draw (3)--(4);\draw (4)--(5);
	\draw  [black,-] (2) to[out=0,in=0] (3);
	\draw [line width=2pt] [black,-] (2) to[out=180,in=180] (3);
	\draw (-4.6cm,-4.6cm) node[inner sep=1.5pt,circle] (7) {$e$};
	\draw (-3.3cm,-4.5cm) node[inner sep=1.5pt,circle] (8) {$e^\prime$};
		\end{tikzpicture} &

		
      \end{tabular}

	\end{center}	
\end{figure}

\subsubsection{Case 3: } Exactly three arcs of the generator
carry a sequence of at least one outgoing arc.
Here, there are two possibilities.
Either both $e$ and $e'$ are among those three arcs
(and those two arcs are symmetric, hence the factor $\frac12$). 
Or, to avoid multiple arcs, we must choose one of $e$ and $e'$ (which are not distinguished from each other), and two additional arcs among the three remaining arcs. 

$$ \frac{3}{2}  {L} \left(\dfrac{{L}}{1-{L}}\right)^3 
+ 3 {L} \left(\dfrac{{L}}{1-{L}}\right)^3 
= \frac{9}{2}  {L} \left(\dfrac{{L}}{1-{L}}\right)^3 $$

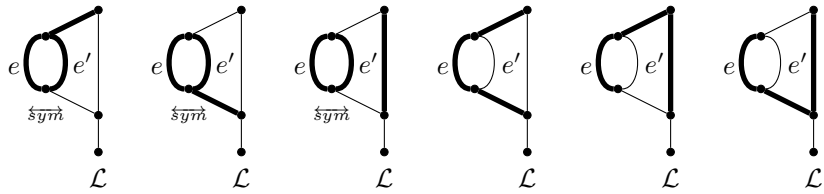
\begin{figure}[!h]
	\begin{center}
	  \begin{tabular}{ccccccc}
		\begin{tikzpicture}[
		scale=0.7,
		level/.style={thick},
		virtual/.style={thick,densely dashed},
		trans/.style={thick,<->,shorten >=2pt,shorten <=2pt,>=stealth},
		classical/.style={thin,double,<->,shorten >=4pt,shorten <=4pt,>=stealth}
		]
	\draw (-3cm,-3.5cm) node[circle,inner sep=1,fill,draw] (1) {};
	\draw (-4cm,-4cm) node[circle,inner sep=1,fill,draw] (2) {};
	\draw (-4cm,-5cm) node[circle,inner sep=1,fill,draw] (3) {};
	\draw (-3cm,-5.5cm) node[circle,inner sep=1,fill,draw] (4) {};
	\draw (-3cm,-6.2cm) node[circle,inner sep=1,fill,draw] (5) {};
	\draw (-3cm,-6.7cm) node[inner sep=1.5pt,circle] (6) {${\mathcal L}$};
	\draw (1)--(2) [line width=2pt] ;\draw (1)--(4);\draw (3)--(4);\draw (4)--(5);
	\draw [line width=2pt] (2) to[out=0,in=0] (3);
	\draw [line width=2pt] (2) to[out=180,in=180] (3);
	\draw (-4.6cm,-4.6cm) node[inner sep=1.5pt,circle] (7) {$e$};
	\draw (-3.3cm,-4.5cm) node[inner sep=1.5pt,circle] (8) {$e^\prime$};
	\draw (-4cm,-5.5cm) node[inner sep=1.5pt,circle] (9) {$\tiny{\overleftrightarrow{{sym}}}$};
		\end{tikzpicture} & 

		\begin{tikzpicture}[
		scale=0.7,
		level/.style={thick},
		virtual/.style={thick,densely dashed},
		trans/.style={thick,<->,shorten >=2pt,shorten <=2pt,>=stealth},
		classical/.style={thin,double,<->,shorten >=4pt,shorten <=4pt,>=stealth}
		]
	\draw (-3cm,-3.5cm) node[circle,inner sep=1,fill,draw] (1) {};
	\draw (-4cm,-4cm) node[circle,inner sep=1,fill,draw] (2) {};
	\draw (-4cm,-5cm) node[circle,inner sep=1,fill,draw] (3) {};
	\draw (-3cm,-5.5cm) node[circle,inner sep=1,fill,draw] (4) {};
	\draw (-3cm,-6.2cm) node[circle,inner sep=1,fill,draw] (5) {};
	\draw (-3cm,-6.7cm) node[inner sep=1.5pt,circle] (6) {${\mathcal L}$};
	\draw (1)--(2);\draw (1)--(4); \draw [line width=2pt]  (3)--(4);\draw (4)--(5);
	\draw [line width=2pt] (2) to[out=0,in=0] (3);
	\draw [line width=2pt] (2) to[out=180,in=180] (3);
	\draw (-4.6cm,-4.6cm) node[inner sep=1.5pt,circle] (7) {$e$};
	\draw (-3.3cm,-4.5cm) node[inner sep=1.5pt,circle] (8) {$e^\prime$};
	\draw (-4cm,-5.5cm) node[inner sep=1.5pt,circle] (9) {$\tiny{\overleftrightarrow{{sym}}}$};
		\end{tikzpicture} &

		\begin{tikzpicture}[
		scale=0.7,
		level/.style={thick},
		virtual/.style={thick,densely dashed},
		trans/.style={thick,<->,shorten >=2pt,shorten <=2pt,>=stealth},
		classical/.style={thin,double,<->,shorten >=4pt,shorten <=4pt,>=stealth}
		]
	\draw (-3cm,-3.5cm) node[circle,inner sep=1,fill,draw] (1) {};
	\draw (-4cm,-4cm) node[circle,inner sep=1,fill,draw] (2) {};
	\draw (-4cm,-5cm) node[circle,inner sep=1,fill,draw] (3) {};
	\draw (-3cm,-5.5cm) node[circle,inner sep=1,fill,draw] (4) {};
	\draw (-3cm,-6.2cm) node[circle,inner sep=1,fill,draw] (5) {};
	\draw (-3cm,-6.7cm) node[inner sep=1.5pt,circle] (6) {${\mathcal L}$};
	\draw (1)--(2);\draw [line width=2pt] (1)--(4);\draw (3)--(4);\draw (4)--(5);
	\draw [line width=2pt] (2) to[out=0,in=0] (3);
	\draw [line width=2pt] (2) to[out=180,in=180] (3);
	\draw (-4.6cm,-4.6cm) node[inner sep=1.5pt,circle] (7) {$e$};
	\draw (-3.3cm,-4.5cm) node[inner sep=1.5pt,circle] (8) {$e^\prime$};
	\draw (-4cm,-5.5cm) node[inner sep=1.5pt,circle] (9) {$\tiny{\overleftrightarrow{{sym}}}$};
		\end{tikzpicture} & 

		\begin{tikzpicture}[
		scale=0.7,
		level/.style={thick},
		virtual/.style={thick,densely dashed},
		trans/.style={thick,<->,shorten >=2pt,shorten <=2pt,>=stealth},
		classical/.style={thin,double,<->,shorten >=4pt,shorten <=4pt,>=stealth}
		]
	\draw (-3cm,-3.5cm) node[circle,inner sep=1,fill,draw] (1) {};
	\draw (-4cm,-4cm) node[circle,inner sep=1,fill,draw] (2) {};
	\draw (-4cm,-5cm) node[circle,inner sep=1,fill,draw] (3) {};
	\draw (-3cm,-5.5cm) node[circle,inner sep=1,fill,draw] (4) {};
	\draw (-3cm,-6.2cm) node[circle,inner sep=1,fill,draw] (5) {};
	\draw (-3cm,-6.7cm) node[inner sep=1.5pt,circle] (6) {${\mathcal L}$};
	\draw [line width=2pt] (1)--(2);\draw (1)--(4);\draw [line width=2pt] (3)--(4);\draw (4)--(5);
	\draw  [black,-] (2) to[out=0,in=0] (3);
	\draw [line width=2pt] (2) to[out=180,in=180] (3);
	\draw (-4.6cm,-4.6cm) node[inner sep=1.5pt,circle] (7) {$e$};
	\draw (-3.3cm,-4.5cm) node[inner sep=1.5pt,circle] (8) {$e^\prime$};
		\end{tikzpicture} &

		\begin{tikzpicture}[
		scale=0.7,
		level/.style={thick},
		virtual/.style={thick,densely dashed},
		trans/.style={thick,<->,shorten >=2pt,shorten <=2pt,>=stealth},
		classical/.style={thin,double,<->,shorten >=4pt,shorten <=4pt,>=stealth}
		]
	\draw (-3cm,-3.5cm) node[circle,inner sep=1,fill,draw] (1) {};
	\draw (-4cm,-4cm) node[circle,inner sep=1,fill,draw] (2) {};
	\draw (-4cm,-5cm) node[circle,inner sep=1,fill,draw] (3) {};
	\draw (-3cm,-5.5cm) node[circle,inner sep=1,fill,draw] (4) {};
	\draw (-3cm,-6.2cm) node[circle,inner sep=1,fill,draw] (5) {};
	\draw (-3cm,-6.7cm) node[inner sep=1.5pt,circle] (6) {${\mathcal L}$};
	\draw [line width=2pt] (1)--(2);\draw [line width=2pt] (1)--(4);\draw (3)--(4);\draw (4)--(5);
	\draw  [black,-] (2) to[out=0,in=0] (3);
	\draw [line width=2pt] (2) to[out=180,in=180] (3);
	\draw (-4.6cm,-4.6cm) node[inner sep=1.5pt,circle] (7) {$e$};
	\draw (-3.3cm,-4.5cm) node[inner sep=1.5pt,circle] (8) {$e^\prime$};
		\end{tikzpicture} &


		\begin{tikzpicture}[
		scale=0.7,
		level/.style={thick},
		virtual/.style={thick,densely dashed},
		trans/.style={thick,<->,shorten >=2pt,shorten <=2pt,>=stealth},
		classical/.style={thin,double,<->,shorten >=4pt,shorten <=4pt,>=stealth}
		]
	\draw (-3cm,-3.5cm) node[circle,inner sep=1,fill,draw] (1) {};
	\draw (-4cm,-4cm) node[circle,inner sep=1,fill,draw] (2) {};
	\draw (-4cm,-5cm) node[circle,inner sep=1,fill,draw] (3) {};
	\draw (-3cm,-5.5cm) node[circle,inner sep=1,fill,draw] (4) {};
	\draw (-3cm,-6.2cm) node[circle,inner sep=1,fill,draw] (5) {};
	\draw (-3cm,-6.7cm) node[inner sep=1.5pt,circle] (6) {${\mathcal L}$};
	\draw (1)--(2);\draw [line width=2pt] (1)--(4);\draw [line width=2pt] (3)--(4);\draw (4)--(5);
	\draw  [black,-] (2) to[out=0,in=0] (3);
	\draw [line width=2pt]  (2) to[out=180,in=180] (3);
	\draw (-4.6cm,-4.6cm) node[inner sep=1.5pt,circle] (7) {$e$};
	\draw (-3.3cm,-4.5cm) node[inner sep=1.5pt,circle] (8) {$e^\prime$};
		\end{tikzpicture} 
		
      \end{tabular}

	\end{center}	
\end{figure}

\subsubsection{Case 4: } Exactly four arcs of the generator
carry a sequence of at least one outgoing arc.
Either both $e$ and $e'$ are among those four arcs
(and those two arcs are symmetric, hence the factor $\frac12$), 
so the last two are chosen among the three other arcs of the generator.
Or we choose the three arcs of the generator 
other than $e$ and $e'$, and $e$ (which is undistinguishable from $e'$).

$$ \frac{\binom{3}{2}}{2} {L} \left(\dfrac{{L}}{1-{L}}\right)^4 
+ {L} \left(\dfrac{{L}}{1-{L}}\right)^4
= \frac{5}{2} {L} \left(\dfrac{{L}}{1-{L}}\right)^4$$

\begin{figure}[!h]
	\begin{center}
	  \begin{tabular}{cccc}
		\begin{tikzpicture}[
		scale=0.7,
		level/.style={thick},
		virtual/.style={thick,densely dashed},
		trans/.style={thick,<->,shorten >=2pt,shorten <=2pt,>=stealth},
		classical/.style={thin,double,<->,shorten >=4pt,shorten <=4pt,>=stealth}
		]
	\draw (-3cm,-3.5cm) node[circle,inner sep=1,fill,draw] (1) {};
	\draw (-4cm,-4cm) node[circle,inner sep=1,fill,draw] (2) {};
	\draw (-4cm,-5cm) node[circle,inner sep=1,fill,draw] (3) {};
	\draw (-3cm,-5.5cm) node[circle,inner sep=1,fill,draw] (4) {};
	\draw (-3cm,-6.2cm) node[circle,inner sep=1,fill,draw] (5) {};
	\draw (-3cm,-6.7cm) node[inner sep=1.5pt,circle] (6) {${\mathcal L}$};
	\draw [line width=2pt](1)--(2) ;\draw (1)--(4);\draw [line width=2pt] (3)--(4);\draw (4)--(5);
	\draw [line width=2pt] (2) to[out=0,in=0] (3);
	\draw [line width=2pt] (2) to[out=180,in=180] (3);
	\draw (-4.6cm,-4.6cm) node[inner sep=1.5pt,circle] (7) {$e$};
	\draw (-3.3cm,-4.5cm) node[inner sep=1.5pt,circle] (8) {$e^\prime$};
	\draw (-4cm,-5.5cm) node[inner sep=1.5pt,circle] (9) {$\tiny{\overleftrightarrow{{sym}}}$};
		\end{tikzpicture} & 

		\begin{tikzpicture}[
		scale=0.7,
		level/.style={thick},
		virtual/.style={thick,densely dashed},
		trans/.style={thick,<->,shorten >=2pt,shorten <=2pt,>=stealth},
		classical/.style={thin,double,<->,shorten >=4pt,shorten <=4pt,>=stealth}
		]
	\draw (-3cm,-3.5cm) node[circle,inner sep=1,fill,draw] (1) {};
	\draw (-4cm,-4cm) node[circle,inner sep=1,fill,draw] (2) {};
	\draw (-4cm,-5cm) node[circle,inner sep=1,fill,draw] (3) {};
	\draw (-3cm,-5.5cm) node[circle,inner sep=1,fill,draw] (4) {};
	\draw (-3cm,-6.2cm) node[circle,inner sep=1,fill,draw] (5) {};
	\draw (-3cm,-6.7cm) node[inner sep=1.5pt,circle] (6) {${\mathcal L}$};
	\draw [line width=2pt](1)--(2);\draw [line width=2pt] (1)--(4); \draw (3)--(4);\draw (4)--(5);
	\draw [line width=2pt] (2) to[out=0,in=0] (3);
	\draw [line width=2pt] (2) to[out=180,in=180] (3);
	\draw (-4.6cm,-4.6cm) node[inner sep=1.5pt,circle] (7) {$e$};
	\draw (-3.3cm,-4.5cm) node[inner sep=1.5pt,circle] (8) {$e^\prime$};
	\draw (-4cm,-5.5cm) node[inner sep=1.5pt,circle] (9) {$\tiny{\overleftrightarrow{{sym}}}$};
		\end{tikzpicture} &

		\begin{tikzpicture}[
		scale=0.7,
		level/.style={thick},
		virtual/.style={thick,densely dashed},
		trans/.style={thick,<->,shorten >=2pt,shorten <=2pt,>=stealth},
		classical/.style={thin,double,<->,shorten >=4pt,shorten <=4pt,>=stealth}
		]
	\draw (-3cm,-3.5cm) node[circle,inner sep=1,fill,draw] (1) {};
	\draw (-4cm,-4cm) node[circle,inner sep=1,fill,draw] (2) {};
	\draw (-4cm,-5cm) node[circle,inner sep=1,fill,draw] (3) {};
	\draw (-3cm,-5.5cm) node[circle,inner sep=1,fill,draw] (4) {};
	\draw (-3cm,-6.2cm) node[circle,inner sep=1,fill,draw] (5) {};
	\draw (-3cm,-6.7cm) node[inner sep=1.5pt,circle] (6) {${\mathcal L}$};
	\draw (1)--(2);\draw [line width=2pt] (1)--(4);\draw [line width=2pt](3)--(4);\draw (4)--(5);
	\draw [line width=2pt] (2) to[out=0,in=0] (3);
	\draw [line width=2pt] (2) to[out=180,in=180] (3);
	\draw (-4.6cm,-4.6cm) node[inner sep=1.5pt,circle] (7) {$e$};
	\draw (-3.3cm,-4.5cm) node[inner sep=1.5pt,circle] (8) {$e^\prime$};
	\draw (-4cm,-5.5cm) node[inner sep=1.5pt,circle] (9) {$\tiny{\overleftrightarrow{{sym}}}$};
		\end{tikzpicture} & 

		\begin{tikzpicture}[
		scale=0.7,
		level/.style={thick},
		virtual/.style={thick,densely dashed},
		trans/.style={thick,<->,shorten >=2pt,shorten <=2pt,>=stealth},
		classical/.style={thin,double,<->,shorten >=4pt,shorten <=4pt,>=stealth}
		]
	\draw (-3cm,-3.5cm) node[circle,inner sep=1,fill,draw] (1) {};
	\draw (-4cm,-4cm) node[circle,inner sep=1,fill,draw] (2) {};
	\draw (-4cm,-5cm) node[circle,inner sep=1,fill,draw] (3) {};
	\draw (-3cm,-5.5cm) node[circle,inner sep=1,fill,draw] (4) {};
	\draw (-3cm,-6.2cm) node[circle,inner sep=1,fill,draw] (5) {};
	\draw (-3cm,-6.7cm) node[inner sep=1.5pt,circle] (6) {${\mathcal L}$};
	\draw [line width=2pt] (1)--(2);\draw [line width=2pt] (1)--(4);\draw [line width=2pt] (3)--(4);\draw (4)--(5);
	\draw  [black,-] (2) to[out=0,in=0] (3);
	\draw [line width=2pt] (2) to[out=180,in=180] (3);
	\draw (-4.6cm,-4.6cm) node[inner sep=1.5pt,circle] (7) {$e$};
	\draw (-3.3cm,-4.5cm) node[inner sep=1.5pt,circle] (8) {$e^\prime$};
		\end{tikzpicture} 
      \end{tabular}

	\end{center}	
\end{figure}

\subsubsection{Case 5: } All five arcs of the generator
carry a sequence of at least one outgoing arc.
The fact that $e$ and $e'$ are symmetric explains the factor $\frac12$.

$$\frac{1}{2} {L} \left(\dfrac{{L}}{1-{L}}\right)^5. $$

\begin{figure}[!h]
	\begin{center}
		\begin{tikzpicture}[
		scale=0.7,
		level/.style={thick},
		virtual/.style={thick,densely dashed},
		trans/.style={thick,<->,shorten >=2pt,shorten <=2pt,>=stealth},
		classical/.style={thin,double,<->,shorten >=4pt,shorten <=4pt,>=stealth}
		]
	\draw (-3cm,-3.5cm) node[circle,inner sep=1,fill,draw] (1) {};
	\draw (-4cm,-4cm) node[circle,inner sep=1,fill,draw] (2) {};
	\draw (-4cm,-5cm) node[circle,inner sep=1,fill,draw] (3) {};
	\draw (-3cm,-5.5cm) node[circle,inner sep=1,fill,draw] (4) {};
	\draw (-3cm,-6.2cm) node[circle,inner sep=1,fill,draw] (5) {};
	\draw (-3cm,-6.7cm) node[inner sep=1.5pt,circle] (6) {${\mathcal L}$};
	\draw [line width=2pt](1)--(2) ;\draw [line width=2pt] (1)--(4);\draw [line width=2pt] (3)--(4);\draw (4)--(5);
	\draw [line width=2pt] (2) to[out=0,in=0] (3);
	\draw [line width=2pt] (2) to[out=180,in=180] (3);
	\draw (-4.6cm,-4.6cm) node[inner sep=1.5pt,circle] (7) {$e$};
	\draw (-3.3cm,-4.5cm) node[inner sep=1.5pt,circle] (8) {$e^\prime$};
	\draw (-4cm,-5.5cm) node[inner sep=1.5pt,circle] (9) {$\tiny{\overleftrightarrow{{sym}}}$};
		\end{tikzpicture} 

	\end{center}	
\end{figure}


\subsection{Exact enumeration formulas}

\subsubsection{Unrooted level-2 networks}

\begin{proposition}\label{eq:exact_unrooted2}
For any $n \geq 1$, the number $u_n$ of unrooted level-2 phylogenetic networks with $(n+1)$ leaves is given by
\begin{equation*}
u_n =   (n-1)! \sum\limits_{{0\leq s \leq q \leq p \leq k \leq i \leq n-1}\atop{{j=n-1-i-k-p-q-s \geq 0}\atop{i\neq 0}}} { {~} \atop {  {{{n+i-1}\choose{i}} {{4i+j-1}\choose{j}} {{i}\choose{k}} {{k}\choose{p}} {{p}\choose{q}} {{q}\choose{s}} \qquad \qquad } \atop {\times  \left(3\right)^i \left(\frac{-15}{6}\right)^k \left(-\frac{16}{15}\right)^p \left(-\frac{1}{2}\right)^q \left(-\frac{3}{16}\right)^{s}.} }}
\end{equation*}
\end{proposition}

\begin{proof}[Sketch]
Recall that $U(z)=z \phi(U(z))$ with  
$\phi (z) = \frac{1}{1-\frac{3z^5-16z^4+32z^3-30z^2+12z}{4(1-z)^4}}$. 
Using first the classical development of $(1-z)^{-n}$ in series (see Eq.~\eqref{eq:classical_expansion}), and then the binomial theorem, we have 
\begin{align*}
\phi(z)^n = & \sum_{i \geq 0} {{n+i-1} \choose {i}} \left( \frac{12z}{4(1-z)^4}+ \frac{- 30z^2 +32z^3-16z^4+3z^5}{4(1-z)^4}\right)^i \\
 = & \sum_{i \geq 0} \sum_{k=0}^i {{n+i-1} \choose {i}} {{i} \choose {k}} \left( \frac{12z}{4(1-z)^4}\right)^{i-k} \left(\frac{- 30z^2 +32z^3-16z^4+3z^5}{4(1-z)^4}\right)^k
\textrm{.}
\end{align*}
We continue applying the binomial theorem inside the above formula, isolating each time the term with the lowest degree in the numerator (that is, first $\tfrac{- 30z^2}{4(1-z)^4}$, second  $\tfrac{32z^3}{4(1-z)^4}$, \dots). This yields 
{\fontsize{8}{9}\selectfont
\begin{align*}
\phi(z)^n = & \sum_{i \geq 0} \sum_{k=0}^i \sum_{p=0}^k \sum_{q=0}^p \sum_{s=0}^q 
{{n+i-1} \choose {i}} {{i} \choose {k}} {{k} \choose {p}} {{p} \choose {q}} {{q} \choose {s}} \\
& \hspace{1cm}
\left( \frac{12z}{4(1-z)^4}\right)^{i-k} 
\left(\frac{- 30z^2}{4(1-z)^4}\right)^{k-p}
\left(\frac{32z^3}{4(1-z)^4}\right)^{p-q}
\left(\frac{-16z^4}{4(1-z)^4}\right)^{q-s}
\left(\frac{3z^5}{4(1-z)^4}\right)^{s}\\
= & \sum_{i \geq 0} \sum_{k=0}^i \sum_{p=0}^k \sum_{q=0}^p \sum_{s=0}^q 
{{n+i-1} \choose {i}} {{i} \choose {k}} {{k} \choose {p}} {{p} \choose {q}} {{q} \choose {s}} 
\frac{
(3)^{i} 
(\tfrac{-15}{6})^{k}
(\tfrac{-16}{15})^{p}
(\tfrac{-1}{2})^{q}
(\tfrac{-3}{16})^{s}}{(1-z)^{4i}} z^{i+k+p+q+s}
\textrm{.}
\end{align*}
}
The result then follows from developing of $(1-z)^{-4i}$ in series as $(1-z)^{-4i} = \sum_{j \geq 0} {{4i+j-1} \choose {j}} z^j$ 
and using the Lagrange inversion formula. 
\end{proof}

\subsubsection{Rooted level-2 networks}

\begin{proposition}\label{exactrooted2}
For any $n \geq 1$, the number $\ell_n$ of rooted level-2 phylogenetic networks with $n$ leaves is given by 
\[
\ell_n =   (n-1)! \sum\limits_{{0 \leq t \leq m  \leq s \leq q \leq p \leq k \leq i \leq n-1}\atop{{j=n-1-i-k-p-q-s-m-t \geq 0}\atop{i\neq 0}}} { {~} \atop {  {{{n+i-1}\choose{i}} {{6i+j-1}\choose{j}} {{i}\choose{k}} {{k}\choose{p}} {{p}\choose{q}} {{q}\choose{s}}{{s}\choose{m}} {{m}\choose{t}}\qquad \qquad } \atop {\times \left(9\right)^i \left(\frac{-17}{6}\right)^{k}  \left(\frac{-53}{34}\right)^p \left(\frac{-148}{159}\right)^q \left(\frac{-81}{148}\right)^s \left(\frac{-8}{27}\right)^m \left(\frac{-1}{8}\right)^t.} }}
\]
\end{proposition}

\begin{proof}[Sketch]
This follows again from the Lagrange inversion formula, 
using the equation ${L}(z)= z\phi( {L}(z))$ for the function $\phi$ given in Theorem~\ref{thm:GF_rooted_level2}. 
The computations involve the usual development of $(1-z)^{-n}$ given by Eq.~\eqref{eq:classical_expansion} and the binomial formula, 
applied following exactly the same steps as in the proof of Proposition~\ref{eq:exact_unrooted2}. 
Details of the computations are left to the reader.
\end{proof}

\end{document}